\documentclass[10pt]{article}
\usepackage[a4paper, margin=1in]{geometry} 

\usepackage{amsmath, amssymb, amsthm}
\usepackage{graphicx}   
\usepackage{hyperref}   
\usepackage{authblk}    
\usepackage{cite}       

\usepackage{subcaption}
\usepackage{comment}
\usepackage{overpic}
\usepackage{amsmath}
\usepackage{amssymb}
\usepackage{xcolor}

\title{On the Rigid-Ruling Folding of Curved Creases: \\
Conjugate-Net Preserving Isometric Deformations of Semi-Discrete Globally Developable Conjugate-Nets}
\author{Klara Mundilova}
\affil{EPFL, Switzerland}
\date{} 

\newcommand{\vecdfn}[1]{\mathbf{#1}} 

\newcommand{\vecn}{{\vecdfn{n}}}

\newcommand{\vecp}{{\vecdfn{p}}}

\newcommand{\vecr}{{\vecdfn{r}}}
\newcommand{\vecs}{{\vecdfn{s}}}
\newcommand{\vect}{{\vecdfn{t}}}

\newcommand{\vecv}{{\vecdfn{v}}}

\newcommand{\vecx}{{\vecdfn{x}}}

\newcommand{\vecB}{{\vecdfn{B}}}

\newcommand{\vecN}{{\vecdfn{N}}}

\newcommand{\vecP}{{\vecdfn{P}}}

\newcommand{\vecR}{{\vecdfn{R}}}
\newcommand{\vecS}{{\vecdfn{S}}}
\newcommand{\vecT}{{\vecdfn{T}}}

\newcommand{\vecX}{{\vecdfn{X}}}

\newcommand{\R}{\mathbb{R}}
\newcommand{\const}{\operatorname{const}}

\newtheorem{lemma}{Lemma}
\newtheorem{corollary}{Corollary}
\newtheorem{theorem}{Theorem}
\newtheorem{definition}{Definition}

\begin{document}

\maketitle

\begin{abstract} 
In this paper, we investigate rigid-ruling folding motions of crease-rule patterns, that is, conjugacy-preserving isometries of developable semi-discrete conjugate nets. We derive two conditions for the rigid-ruling foldability of pairs of curves and consider two applications. First, we introduce computations that enable the sequential construction of rigid-ruling foldable crease-rule patterns. Second, we examine combinations of planar and constant fold-angle creases. In particular, we show that constant fold-angle creases are only compatible with other constant fold-angle creases, and we provide a characterization of rigid-ruling foldable combinations of planar and constant fold-angle creases.
\end{abstract}

\section{Introduction}\label{sec:intro}

Structures made from sheet materials offer advantages, such as cost-effective fabrication while enabling the creation of intricate forms by joining sheets along curved boundaries. Beyond their artistic appeal, such structures have found applications in architectural shells~\cite{pottmann2008}, innovative furniture design~\cite{wills2006d}, and hydrodynamic ship hulls~\cite{PEREZARRIBAS2006584}.
A notable subclass, \emph{curved-crease origami}, arises when sheet material is folded along curves. Initially explored in artistic contexts~\cite{koschitz2014computational, demaine2015review}, curved-crease origami shapes have since evolved into an active field of both theoretical investigation~\cite{demaine2015characterization} and applied research of transformable structures~\cite{tachi2013composite}.

Mathematically, shapes obtained by bending a sheet of paper are (parts of) developable surfaces. As a subclass of ruled surfaces, developable surfaces are composed of families of straight lines, called \emph{rulings}, and are characterized by a constant tangent plane along each ruling~\cite{linegeometry}. 
Compositions of developable patches joined along curved boundaries constitute semi-discrete structures that can approximate curved geometries, see for example~\cite{stein2018} and references therein. Curved-crease origami comprises a notable subclass, in which the common boundary curve, the \emph{curved crease}, has the (geometric) property that, when unrolled into the plane with respect to either incident patch, it coincides, while the two patches lie on opposite sides of the developed curve.

In recent years, compositions of developable patches have been studied both in the theory of curved-crease origami and in semi-discrete differential geometry, leading to two terminologies for same concepts. In the following, we provide a brief overview of these perspectives.

\paragraph{Curved-crease origami.} In the study of curved-crease origami, a fundamental research question is whether a given folded shape “exists” mathematically, that is, whether it can be described using developable patches~\cite{demaine2011non}. This question is nontrivial, since determining the rulings is itself challenging. One approach is to start from the 2D configuration prescribed by creases and rulings, and ask whether the given \emph{crease-rule pattern} admits a \emph{folded state}, that is, a semi-discrete structure composed of developable patches that is isometric to the 2D configuration with the same rulings (and not planar).

In general, given two 2D developable patches with prescribed rulings, joining them along a smooth boundary curve typically results in a one-parameter family of corresponding 3D configurations~\cite{mundilovaphd}. A smooth variation between these configurations can be interpreted as a folding motion that preserves the rulings. In contrast, joining three patches with prescribed rulings is usually overconstrained, and a 3D configuration may not exist. Nevertheless, there are examples in which three or more ruled 2D patches admit a continuous family of 3D states, allowing the angles between the patches to vary smoothly. We refer to this type of motion as a \emph{rigid-ruling (folding) motion}. 
Analogous considerations apply to the special case of curved-crease origami, where crease-rule patterns involving two or more creases admit a rigid-ruling folding motion only in special cases (see Figure~\ref{fig:introa}).

\paragraph{Semi-discrete differential geometry.} Smooth conjugate nets are surface parametrizations characterized by the property that their mixed partial derivatives are tangent to the surface (or, in some cases, vanish). Intuitively, such parametrizations satisfy the following geometric property: the rulings of the envelope of tangent planes along one family of parameter curves are tangent to the corresponding curves of the incident second family. The discrete counterparts of conjugate nets are regular planar quad (PQ) meshes, where the edges of the polylines correspond to the conjugate directions~\cite[§27]{peterson1868kurven} (see Figure~\ref{fig:introb}). One-directional refinement of PQ meshes yields compositions of developable strips, see for example~\cite[§1.3]{muller2015semi}, which correspond to semi-discrete conjugate nets, with ruling polylines and common boundary curves representing the conjugate families. Curved-crease origami shapes form a special subclass of semi-discrete conjugate nets, namely those that are \emph{globally} developable.

Classical differential geometry investigates isometric deformations of smooth surfaces, with those that preserve conjugate nets forming a special subclass. 
As early as 1860, Bianchi formulated a necessary and sufficient condition for a smooth surface to admit an isometric deformation that maps conjugate nets to conjugate nets~\cite[§141]{eisenhart2013treatise}. 

In the discrete setting, isometric deformations of conjugate nets correspond to \emph{flexible} (or \emph{rigidly-foldable}) PQ meshes, that is, meshes that can be continuously transformed by a change of the dihedral angles only. Building on previous work, such as~\cite{stachel2010kinematic}, only recently have all $3 \times 3$ flexible quad meshes been fully classified by Izmestiev~\cite{izmestiev}. While a quad mesh is flexible if and only if all of its $3 \times 3$ submeshes are flexible~\cite[Theorem 3.2]{schief2008integrability}, assembling such submeshes into larger flexible quad meshes is nontrivial, and to date only special cases have been investigated. A complete characterization of all flexible quad meshes remains an open problem.

\begin{figure}[t]

\begin{subfigure}[b]{0.49\textwidth}
\centering
\includegraphics[width=\textwidth]{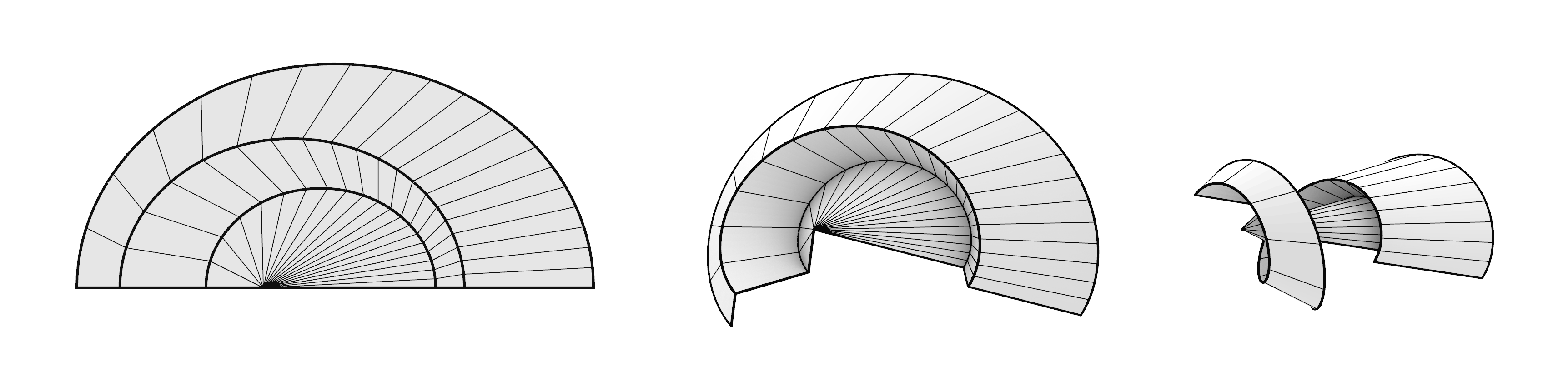}
    \caption{Rigid-ruling folding motion of a crease-rule pattern.}\label{fig:introa}
\end{subfigure}
\begin{subfigure}[b]{0.49\textwidth}
\centering
\includegraphics[width=\textwidth, trim = -3cm 0cm 0cm 0cm, clip]{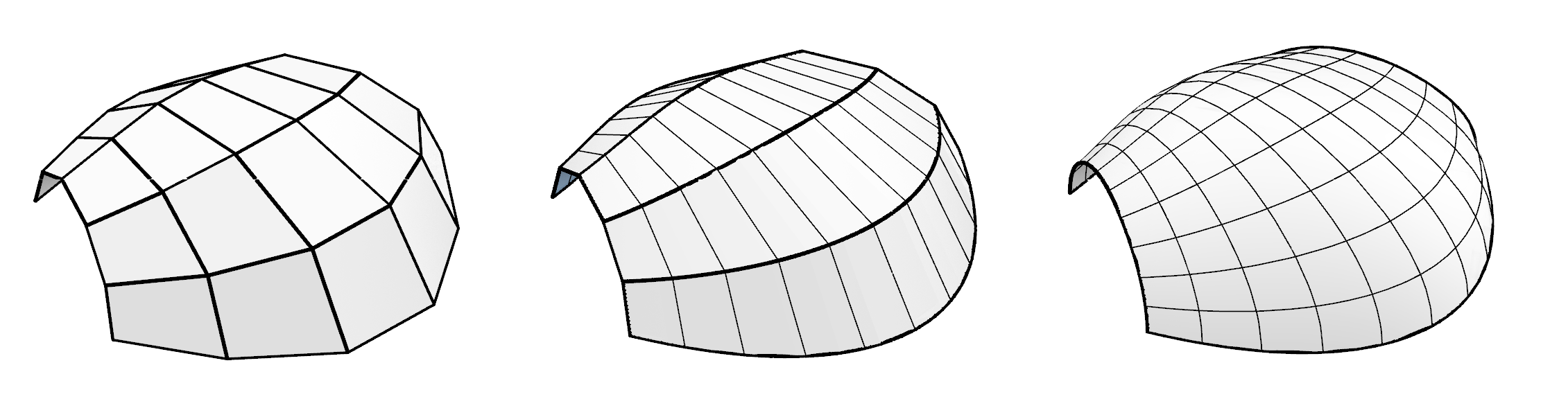}
    \caption{Discrete, semi-discrete, and smooth conjugate nets.}\label{fig:introb}
\end{subfigure}

    \caption{Illustration of concepts discussed in Section~\ref{sec:intro}.}
\end{figure}

\paragraph{} A folding motion of a combination of developable patches that preserves the rulings is an isometry that also preserves conjugate nets. Consequently, rigid-ruling folding motions of compositions of developable patches correspond to conjugate-net-preserving isometries of semi-discrete conjugate nets. 
This raises the question of which combinations of curves and rulings are compatible, and whether these semi-discrete configurations admit smooth or fully discrete counterparts. To address this, we review relevant results on subclasses of conjugate nets:
\begin{itemize}
\item \emph{T-nets:} A prominent class of conjugate nets are those whose conjugate families lie in mutually perpendicular planes, known as T-nets. Their discrete counterparts, called \emph{T-hedra}, are characterized by trapezoidal faces, which give this class its name. Discrete T-nets were first observed by Graf and Sauer~\cite{sauer1931flachenverbiegung} to be flexible, and recent work has provided smooth analogs as well as synthetic descriptions of the smooth and semi-discrete cases~\cite{izmestiev2024isometric}. Recently, T-hedra have been applied in the design of tubular structures~\cite{sharifmoghaddam2023generalizing}.

In the context of origami, creases that remain planar during the folding motion form a special class and can be generated by reflection. Nawratil~\cite{nawratil2023isometrically} characterizes all pairs of polylines on discrete cylinders and cones that remain planar under deformation. A consequence of this study is, that not all rigidly foldable discrete planar crease patterns correspond to T-hedra. In contrast, in the semi-discrete case, Mundilova and Nawratil~\cite{mundilova2024planar} show that every combination of planar creases admitting a rigid-ruling folding motion are semi-discrete T-nets.

\item \emph{V-nets:} Another prominent subclass of conjugate nets that are simultaneously geodesic nets was first studied by Voss in 1889 and is now known as V-nets~\cite{voss1889ueber}. In the discrete setting, the geodesic condition translates into the equality of opposite sector angles at each vertex~\cite[Sec.~12.2]{sauer1970differenzengeometrie}. Discrete V-nets are closely related to anti-V-nets, which are also studied in the context of flat-foldable origami with valence four vertices. Due to their applications in transformable structures, both discrete V-nets and anti-V-nets have been the subject of research, resulting in several design approaches; see for example~\cite{kilian2024interactive} and the references therein.

A sequence of flat-foldable developable vertices connected along their major crease maintains a constant dihedral angle. Consequently, in the semi-discrete setting, anti-V-nets correspond to constant fold-angle creases. Inspired by David Huffman's artistic exploration of conical curved-creases, Demaine et al.~\cite{demaine2018conic} characterize all constant-fold angle creases between cylinders and cones. In subsequent work, Demaine et al.~\cite{demaine2023locally} show that such patterns admit two types of discretization that are rigidly foldable and highlight conditions under which these discretizations are even flat-foldable. A publication on smooth, discrete, and semi-discrete V-nets is currently being prepared by Izmestiev et al.~\cite{izmestiev2025}.

\item \emph{Cone-nets:} The discrete conic-crease patterns are examples of discrete cone-nets, specifically, PQ meshes composed of strips of cylinders or cones. Moreover, they represent a subclass of \emph{axial} cone-nets, as their cone apices (or points at infinity in the case of cylinders) remain collinear throughout the folding motion. Recently, building on this work, Nawratil~\cite{nawratil2024axial} characterizes all flexible axial discrete cone-nets.
\end{itemize}

\paragraph{Contributions.} In this paper, we study the rigid-ruling foldability of \emph{regular} crease-rule patterns, that is, crease patterns consisting of a sequence of patches separated by creases. Equivalently, this amounts to investigating conjugate-net-preserving isometries of globally developable semi-discrete conjugate nets. After reviewing our notation and introducing preliminary results in Section~\ref{sec:primer}, the contributions of this paper are as follows:
\begin{itemize} 
\item[(1)] In Section~\ref{sec:mainresult}, we identify conditions under which a given crease-rule pattern with two creases admits a rigid-ruling folding motion.

\item[(2)] In Section~\ref{sec:appending}, we show how to successively construct rigid-ruling foldable crease patterns by adding curved creases and patches in a way that preserves the rigid-ruling folding motion.

\item [(3)] Finally, in Section~\ref{sec:characterization}, we analyze the rigid-ruling folding compatibility of two special classes of creases: planar creases and constant fold-angle creases.
\end{itemize}

With this work, our goal is to make a step towards understanding conjugate-net preserving isometries of semi-discrete conjugate-nets.

\section{Preliminaries}\label{sec:primer}

For the sake of self-containedness, we include the notation used in our analysis of compositions of developable surfaces, as presented in~\cite{mundilovaphd, mundilova2024planar}, which builds on the work of~\cite{demaine2015characterization, demaine2018conic}. For further background on ruled and developable surfaces, we refer the reader to~\cite{linegeometry}.

\subsection{Developable surfaces}

In this paper, we parametrize developable surfaces
as ruled surfaces while imposing an additional constraint to ensure developability.

\subsubsection{Ruled surfaces}

Recall that a \emph{ruled surface} can be parametrized as
\begin{equation}\label{eqn:ruledpatch}
\vecS(t,u) = \vecX(t) + u \vecR(t), 
\end{equation}
where $\vecX(t): T \to \mathbb{R}^3$ is a curve, the \emph{directrix}, and $\vecR(t) : T \to S^2$ are unit-length vectors, the so-called \emph{ruling directions}. Without loss of generality, we assume $T = [0, t_{\max}]$, for some $t_{\max} > 0$, and $u \in \R$. Additionally, for the surface to be $C^2$, we require both $\vecX(t)$ and $\vecR(t)$ to be $C^2$.

We assume that the curve is equipped with an orthonormal frame and we describe the location of the ruling vectors with respect to this frame. To ensure that the curve's frame is continuous, we specify the curve $\vecX(t)$ through three functions: $K(t) : T \to \R$, 
   $\tau(t) : T \to \R$, and $s(t) : T \to \R$. These functions define $\vecX(t)$ up to Euclidean motion through the Frenet-Serret formulas, that is, $\vecX'(t) = s'(t) \vecT(t)$, where  \begin{equation}\label{eqn:frenetserret}
      \begin{pmatrix}
      \vecT'(t) \\
      \vecN'(t) \\
      \vecB'(t)
      \end{pmatrix}
      = s'(t) \begin{pmatrix}
          0 & K(t) & 0 \\
          -K(t) & 0 & \tau(t) \\
          0 & -\tau(t) & 0 
      \end{pmatrix}
      \begin{pmatrix}
      \vecT(t) \\
      \vecN(t) \\
      \vecB(t)
      \end{pmatrix},
      \end{equation}
      see \cite{nutbourne1988differential}.
Here we require that 
      the $K(t)$, $\tau(t)$, and $s'(t)$ are continuous, and $s'(t) > 0$. 

Note that we allow $K(t)$ to take negative values. Moreover, isolated points or intervals where $K(t) = 0$ or $\tau(t)= 0$ still yield a continuous frame $(\vecT(t), \vecN(t), \vecB(t))$. The described frame is not a Frenet-frame. However, at parameter values where the Frenet-frame is defined, the computed frame coincides with the Frenet-frame, differing only by sign. Moreover, $K(t)$ corresponds to the curvature of the directrix up to sign, while $\tau(t)$ is the torsion of the directrix when defined. In the following, we refer to $K(t)$ as the \emph{(signed) curvature} and to $\tau(t)$ as the \emph{torsion}. Furthermore, $s(t)$ denotes the \emph{arc-length} of the directrix, and $s'(t)$ represents the \emph{parametrization speed}.
        
To determine the ruling directions with respect to the frame $\left(\vecT(t), \vecN(t), \vecB(t)\right)$, we introduce two additional angular functions: the \emph{inclination angle} $\varphi(t) : T \to \R$ and the \emph{ruling angle} $\theta(t): T \to \R$; see Figure~\ref{fig:ill_not_patch}.
  
The inclination angle $\varphi(t)$ encodes the angle between a one-parameter family of planes $\Pi(t)$, which contain the curve's tangent vectors $\vecT(t)$. Those planes will correspond to the tangent planes if the ruled surface is developable. 
We express the normal vector $\vecP(t)$ of $\Pi(t)$ as 
\begin{align}
      \vecP(t) = \cos \varphi(t) ~\vecB(t) +  \sin \varphi(t) ~ \vecN(t),\label{eqn:tanplanenormal}
 \end{align}  
 resulting in $\varphi(t)$ being the signed angle between $\vecP(t)$ and $\vecB(t)$.

  Within the plane $\Pi(t)$, we locate the ruling direction using
  the ruling angle as 
  \begin{align}
      \vecR(t) &= \cos \theta(t) ~\vecT(t) + \sin \theta(t) \left(\vecP(t) \times \vecT(t)\right) \nonumber \\
      &= \cos \theta(t) ~\vecT(t) + \sin \theta(t) \left(\cos \varphi(t) ~ \vecN(t) - \sin\varphi(t) ~\vecB(t)\right).\label{eqn:Rexpl}
  \end{align} 
 For the ruling direction to be $C^2$, we require both $\theta(t)$ and $\varphi(t)$ to be $C^2$. Note that for the subsequent computations, $C^1$ would be sufficient. 
 
 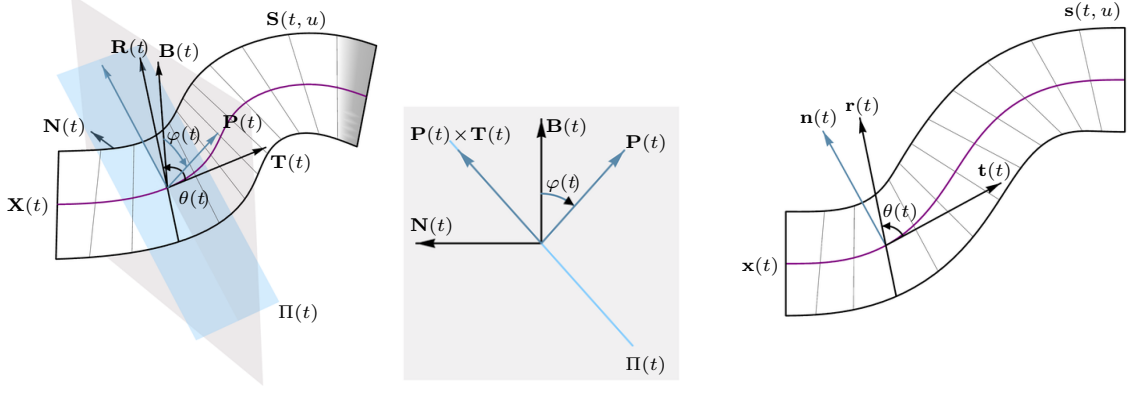
\begin{figure}
 \centering
      \begin{scriptsize}%
\begin{overpic}[height=0.33\textwidth,trim = 2cm 0.1cm 0.7cm 1.1cm, clip]{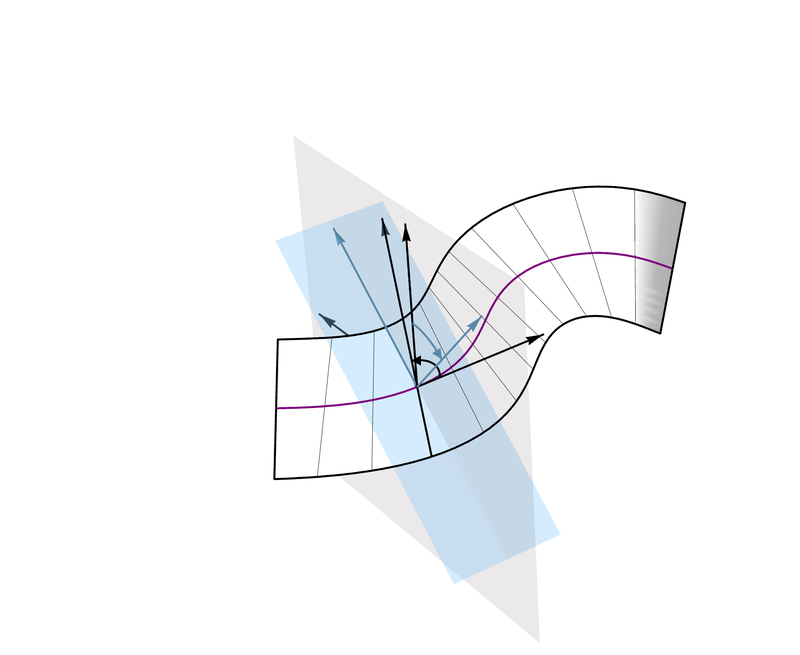}%
\put(60,91){$\vecS(t,u)$}
\put(-5,44){$\vecX(t)$}
\put(61,55){$\vecT(t)$}
\put(4,64){$\vecN(t)$}
\put(33,83){$\vecB(t)$}
\put(49,65){$\vecP(t)$}
\put(38,46){$\theta(t\!)$}
\put(35,62){$\varphi(t\!)$}
\put(21,84){$\vecR(t\!)$}
\put(63,17){$\Pi(t)$}
\end{overpic}%
\begin{overpic}[height=0.33\textwidth,trim = -0cm -0.cm -0.4cm -1cm, clip]{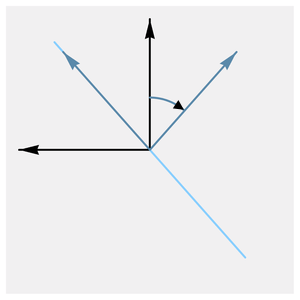}
\put(3,62){$\vecP(t)\!\times\!\vecT(t)$}
\put(3,39){$\vecN(t)$}
\put(37,64){$\vecB(t)$}
\put(57,60){$\vecP(t)$}
\put(37,50){$\varphi(t\!)$}
\put(57,4){$\Pi(t)$}
\end{overpic}%
\begin{overpic}[height=0.33\textwidth,trim = -0.4cm 0.2cm 0.cm -0.2cm, clip]{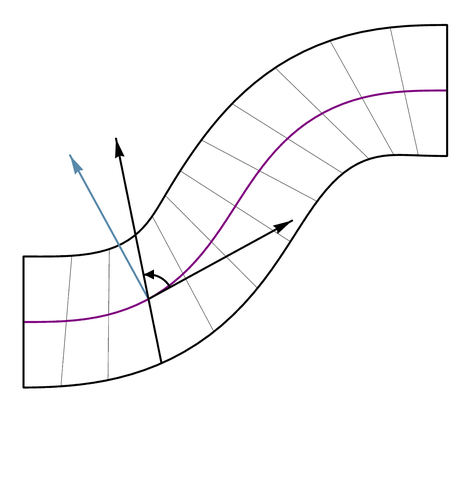}%
  \put(3,28){$\vecx(t)$}
  \put(81,90){$\vecs(t,u)$}
  \put(60,51){$\vect(t)$}
  \put(37,41){$\theta(t)$}
  \put(17,64){$\vecn(t)$}
  \put(28,67){$\vecr(t)$}
\end{overpic}%
    \end{scriptsize}%
     \caption{Illustration of the notation of a developable patch and its development (as in~\cite{mundilovaphd, mundilova2024planar}).}\label{fig:ill_not_patch}
      \end{figure}

\subsubsection{Developability condition}\label{sec:development}

  It is known that the ruled surface in Equation~\eqref{eqn:ruledpatch} is developable if for all rulings, the tangent planes along points on a ruling are the same~\cite{linegeometry}. This condition can be expressed as 
\begin{equation}\label{eqn:devgen}
  \det\left(\vecX'(t), \vecR(t), \vecR'(t) \right) = 0.
  \end{equation}
  Using Equation~\eqref{eqn:frenetserret} and Equation~\eqref{eqn:Rexpl} for ruled surfaces with planar directrices, this condition simplifies to 
\begin{equation}\label{eqn:dev1}
\frac{\varphi'(t)}{s'(t)} = \tau(t) + K(t) \sin\varphi(t) \cot \theta(t); 
\end{equation}
see~\cite{mundilovaphd} for more details.

Given a developable surface $\vecS(t,u)$, we will parametrize its flattened configuration, the \emph{development}, by 
$\vecs(t,u) = \vecx(t) + u \vecr(t)$, where $\vecx(t) : T \to \R^2$ represents the 2D counterpart of the directrix $\vecX(t)$, and $\vecr(t) : T \to S^1$ the unit-length 2D ruling direction; see Figure~\ref{fig:ill_not_patch}.

 To obtain the developed directrix $\vecx(t)$, we consider the \emph{geodesic curvature} of $\vecX(t)$ as a curve on $\vecS(t,u)$, that is, the curvature of the projection of $\vecX(t)$ on $\Pi(t)$ at parameter $t$,
  \begin{equation}\label{eqn:curv1}
  k(t) = K(t) \cos \varphi(t).
  \end{equation} 
We obtain $\vecx(t)$ by determining the 2D curve with signed curvature $k(t)$ and parame\-trization speed $s(t)$. This amounts in solving the system of differential equations $\vecx'(t) = s'(t) \vect(t)$, where 
 \begin{equation}\label{eqn:frenetserret2D}
      \begin{pmatrix}
      \vect'(t) \\
      \vecn'(t) 
      \end{pmatrix}
      = s'(t) \begin{pmatrix}
          0 & k(t)  \\
          -k(t) & 0  
      \end{pmatrix}
      \begin{pmatrix}
      \vect(t) \\
      \vecn(t) 
      \end{pmatrix}.
      \end{equation}
As isometry preserves angles on surfaces, particularly the oriented angle between $\vecT(t)$ and $\vecR(t)$ or $\vecP(t) \times \vecT(t)$, the developed ruling directions read 
$$
\vecr(t) = \cos \theta(t) ~\vect(t) + \sin \theta(t) ~\vecn(t).
$$

\subsubsection{Ruling curvature}\label{sec:rulingcurv}

In subsequent sections, we will consider bent configurations $\vecS(t,u)$ of a planar developable patch $\vecs(t,u)$. To show that two such deformations are identical, we calculate the curvature that indicates the surface's bend perpendicular to the rulings, following the methodology of~\cite{demaine2015characterization, demaine2018conic}. This \emph{ruling curvature} is the normal curvature at a given point on an arc-length parametrized curve perpendicular to the rulings at parameter $t$, expressed as:
\begin{equation}\label{eqn:rulcurv}
V(t) = s'(t) K(t) \sin\varphi(t) \frac{1}{\sin\theta(t)} = s'(t) k(t) \tan\varphi(t) \frac{1}{\sin\theta(t)}
\end{equation}

We make note of the following property, which directly follows form Equation~\eqref{eqn:rulcurv}, Equation~\eqref{eqn:curv1}, and Equation~\eqref{eqn:dev1}.

\begin{lemma}\label{lem:rulcurv}
    The ruling curvature determines the bend configuration of a developed patch up to Euclidean motion.
\end{lemma}

\subsection{Folded states of crease-rule patterns}

In this section, we discuss the differential and algebraic equations that govern the computation of folded states of a given crease-rule pattern.

\subsubsection{Geometry of a single crease}\label{sec:singlecrease}

We begin by reviewing the well-known computation of folded states for a single crease, where the rulings of the two adjacent patches are prescribed~\cite{fuchs1999more, demaine2015characterization}.

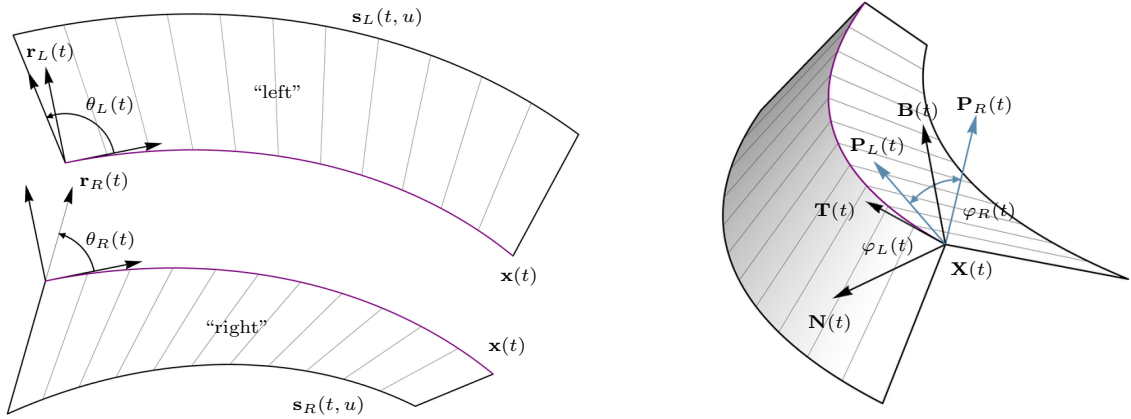
\begin{figure}[t]
  \centering
  \begin{scriptsize}
\begin{overpic}[width=0.49\textwidth,trim = 0cm 0cm 0cm 0.cm, clip]{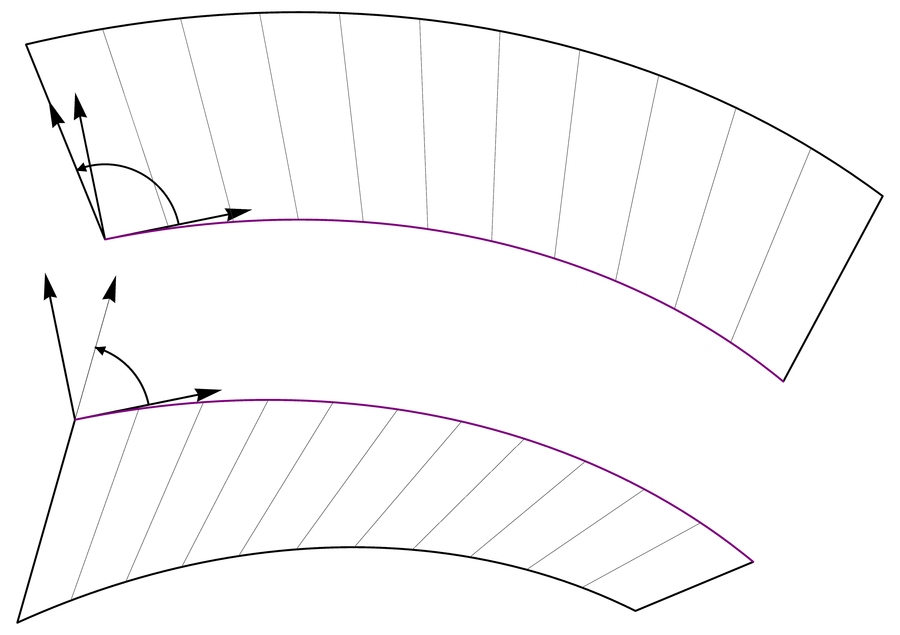}
        \put(15,53){$\theta_{L}(t)$}
        \put(85,24){$\vecx(t)$}
        \put(5,62){$\vecr_{L}(t)$}
        \put(60,68){$\vecs_L(t,u)$}

        \put(15,30){$\theta_{R}(t)$}
        \put(83,12){$\vecx(t)$}
        \put(14,40){$\vecr_R(t)$}
        \put(50,2){$\vecs_R(t,u)$}
       
        \put(43,55){``left''}
        \put(35,15){``right''}
    \end{overpic}
  \end{scriptsize}
    \begin{scriptsize}
        \begin{overpic}[width=0.49\textwidth,trim = 1cm 0.5cm 0.5cm 0.5cm, clip]{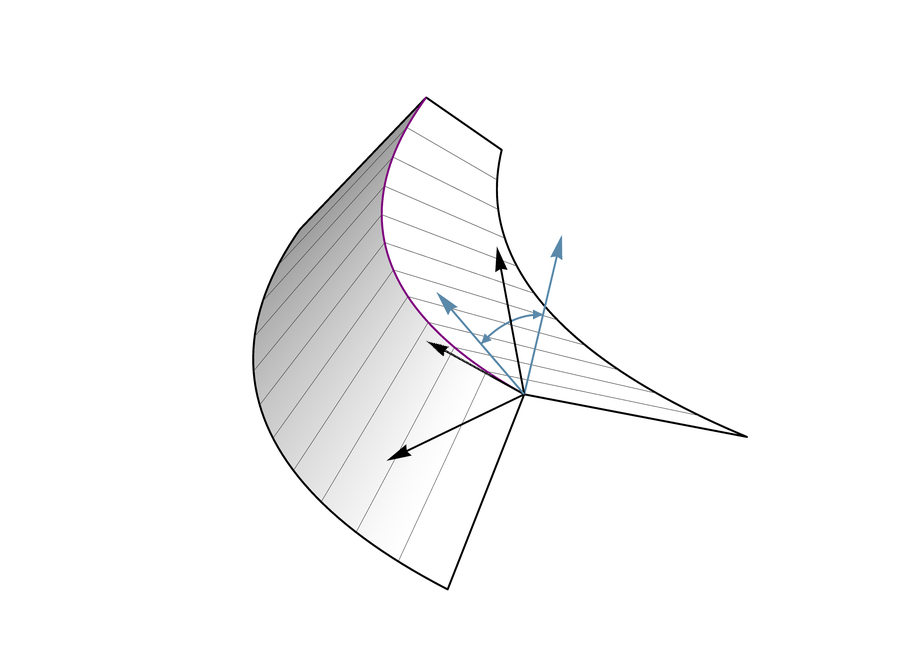}   
        \put(57,25){$\vecX(t)$}
        \put(59,35){$\varphi_{R}(t)$}
        \put(42,29){$\varphi_{L}(t)$}
        \put(48,52){$\vecB(t)$}
        \put(34,35){$\vecT(t)$}
        \put(33,16){$\vecN(t)$}
        \put(58,53){$\vecP_R(t)$}
        \put(40,46){$\vecP_L(t)$}
        \end{overpic}%
        \end{scriptsize}%
  \caption{Illustration of the notation introduced in Section~\ref{sec:singlecrease} (as in~\cite{mundilovaphd, mundilova2024planar}).}\label{fig:singlecrease}
\end{figure}

 A curved crease, denoted as $\vecx(t)$, locally divides the sheet into two sides, a ``left'' and ``right'' side with respect to the orthonormal frame of $\vecx(t)$, 
\begin{align*}
\vecs_L(t,u) = \vecx(t) + u \vecr_L(t)
&&    \text{and} &&
    \vecs_R(t,u) = \vecx(t)- u \vecr_R(t),   
\end{align*}
where 
$\vecr_j(t) = \cos \theta_j(t) ~\vect(t) + \sin \theta_j(t) ~ \vecn(t)$ for $j \in \{L,R\}$ are the (left-side) ruling directions and $\left(\vect(t), \vecn(t)\right)$ the local orthonormal frame; see Figure~\ref{fig:singlecrease}.
Let $k(t)$ and $s(t)$ denote the curvature and arc-length of the crease curve $\vecx(t)$.

To compute the corresponding 3D configuration, that is, 
\begin{align*}
\vecS_L(t,u) = \vecX(t) + u \vecR_L(t)
&&    \text{and} &&
    \vecS_R(t,u) = \vecX(t)- u \vecR_R(t),   
\end{align*}
we determine the curvature $K(t)$ and torsion $\tau(t)$ of the 3D crease and the inclination angles $\varphi_L(t)$ and $\varphi_R(t)$ of the adjacent surfaces.

Since Equation~\eqref{eqn:curv1} applies to both inclination angles, it follows that $\cos \varphi_L(t) = \cos\varphi_R(t)$. The interesting ``folded'' case occurs when 
$\varphi(t) = \varphi_L(t) = -\varphi_R(t)$~\cite{fuchs1999more}.
In this case, we consider the developability condition (Equation~\eqref{eqn:dev1}) for both incident surfaces and Equation~\eqref{eqn:curv1}. Solving for $\varphi'(t)$, $\tau(t)$, and $K(t)$ results in  
\begin{align}
\varphi'(t) &= \frac{1}{2} s'(t) k(t) \left(\cot \theta_R(t) + \cot \theta_L(t)\right) \tan \varphi(t),\label{eqn:phistrich}\\
\tau(t) &= \frac{1}{2} s'(t) k(t)  \left(\cot \theta_R(t) - \cot \theta_L(t)\right) \tan \varphi(t),\label{eqn:tau} \\
K(t) &= \frac{k(t)}{\cos\varphi(t)}.\label{eqn:kappa}
\end{align}

The function $2 \varphi(t)$ quantifies the deviation from a configuration in which the tangent planes are aligned, indicating a state where the paper is locally uncreased. Consequently, $\varphi(t)$ represents \emph{half of the fold-angle}. Additionally, it is important to note that $\varphi(t)$ is determined by the differential equation in Equation~\eqref{eqn:phistrich} up to the initial value. In fact, the inclination can be obtained as
\begin{align}\label{eqn:phi}
\varphi(t) = \arcsin\left(c_0 e^{\int_0^t f(\bar t) d\bar t}\right), &&  \text{where} && 
f(t) = \frac{1}{2} s'(t) k(t) \left(\cot \theta_L(t) + \cot \theta_R(t)\right)
\end{align}
and $c_0 = \sin \varphi(0)$ is an appropriate initial value. For $\varphi(t)$ to be real-valued for all $t \in T$, we require 
$$
\left|c_0 \right|\leq c_{\max} = \min_{t\in T} e^{-\int_0^t f(\bar t)d \bar t}.
$$
Upon successful computation of the inclination angle $\varphi(t)$, the curvature and torsion of $\vecX(t)$ follow from Equation~\eqref{eqn:tau} and Equation~\eqref{eqn:kappa}. 

The 3D configuration is then obtained by solving the Frenet-Serret equations in Equation~\eqref{eqn:frenetserret} and constructing the corresponding 3D ruling directions $\vecR_L(t)$ and $\vecR_R(t)$ (Equation~\eqref{eqn:Rexpl}) for appropriate inclination and ruling angles\footnote{Since typically $c_{\max} > 0$, there generally exists a one-parameter family of suitable fold-angles in the vicinity of the flat state.}.

There are the following two special cases, corresponding to the cases where the ruling angles are equal or sum up to $\pi$:
\begin{itemize}
\item \emph{Planar creases:} If $\theta_L(t) = \theta_R(t)$, it follows from Equation~\eqref{eqn:tau} that the torsion vanishes, $\tau(t) = 0$, and the crease becomes a planar crease. In general, a ``maximally folded state'' occurs when one tangent plane becomes perpendicular to the plane containing the curve.
\item \emph{Constant fold-angle creases:} If 
$\theta_L(t) = \pi-\theta_R(t)$, it follows from Equation~\eqref{eqn:phistrich} that the first derivative of the inclination angle, 
$\varphi'(t) = 0$, vanishes, indicating that the crease maintains a constant inclination angle and, consequently, a constant fold-angle along the crease. Since $\varphi(t) = \arcsin(c_0)$ is real-valued for $c_0 \in \left[-\frac{\pi}{2},\frac{\pi}{2}\right]$, the full folding motion exists. In the discrete case, this crease corresponds to a sequence of flat-foldable vertices with the same assignment along the discretized crease. The fully flat-folded state of the smooth configuration corresponds to a completely ``rolled-up'' configuration.
\end{itemize}

In addition, we regard creases along straight segments as both planar and constant fold-angle creases, regardless of their incident rulings.

\subsubsection{Crease-rule patterns }

The goal of this paper is to analyze conjugate-net–preserving isometries of conjugate nets obtained through one-directional refinement of 
PQ meshes. While artistic curved-crease origami designs may arrange creases in many different ways, we restrict our attention to those formed by a sequential arrangement. 

\begin{definition}
    A \emph{regular crease-rule pattern} $\mathcal{P}$ is a sequence of non-intersecting planar smooth crease curves $\vecx_1(t)$, \dots, $\vecx_n(t): T \to \R^2$, together with two boundary curves $\vecx_0(t), \vecx_{n+1}(t) : T \to \R^2$, such that 
    \begin{itemize}
        \item Each $\vecx_i(t)$ lies to the left of $\vecx_{i-1}(t)$, for all $i \leq n+1$, and 
        \item The interiors of the patches
    $$
    \vecp_{i,R}(t,u) = (1-u) \vecx_i(t) + u \vecx_{i+1}(t) \quad \text{ with } \quad (t,u) \in T\times (0,1),
    $$
   are disjoint, their rulings are not aligned with the tangents of the boundary curves, and the boundary curves are free of singularities of the patches.
    \end{itemize} 
\end{definition}

In the following, we will denote by $s_i(t)$ the curve's $\vecx_i(t)$ arc-lengths and by $k_i(t)$ their curvatures. The (left-side) ruling directions of patches bounded by two creases are 
 \begin{equation}\label{eqn:rulingdir}
\vecr_{i,L}(t) =\vecr_{i+1,R}(t) = \frac{\vecx_{i+1}(t) - \vecx_i(t)}{\left|\vecx_{i+1}(t) - \vecx_i(t)\right|}.
\end{equation}
Moreover, the ruling angles of the patches that are bounded by two creases can then be obtained from 
\begin{align}
    \theta_{i,j}(t) &= 
    \arctan\left(\vect_{ij}(t) \cdot \vecr_{ij}(t), 
    \vecn_{ij}(t) \cdot \vecr_{ij}(t)\right),\label{eqn:rulingangle1}
\end{align}
where $j\in \{L,R\}$. 
Here, $\vect_{i}(t) $ represents the unit tangent, while $\vecn_i(t)$ stands for the (left-side) normal of 
$\vecx_{i}(t)$, respectively.

\subsubsection{Computation of folded states of crease-rule patterns}\label{sec:multiplecrease}

\begin{figure}[t]
  \centering
  \begin{scriptsize}
        \begin{overpic}[width=0.49\textwidth]{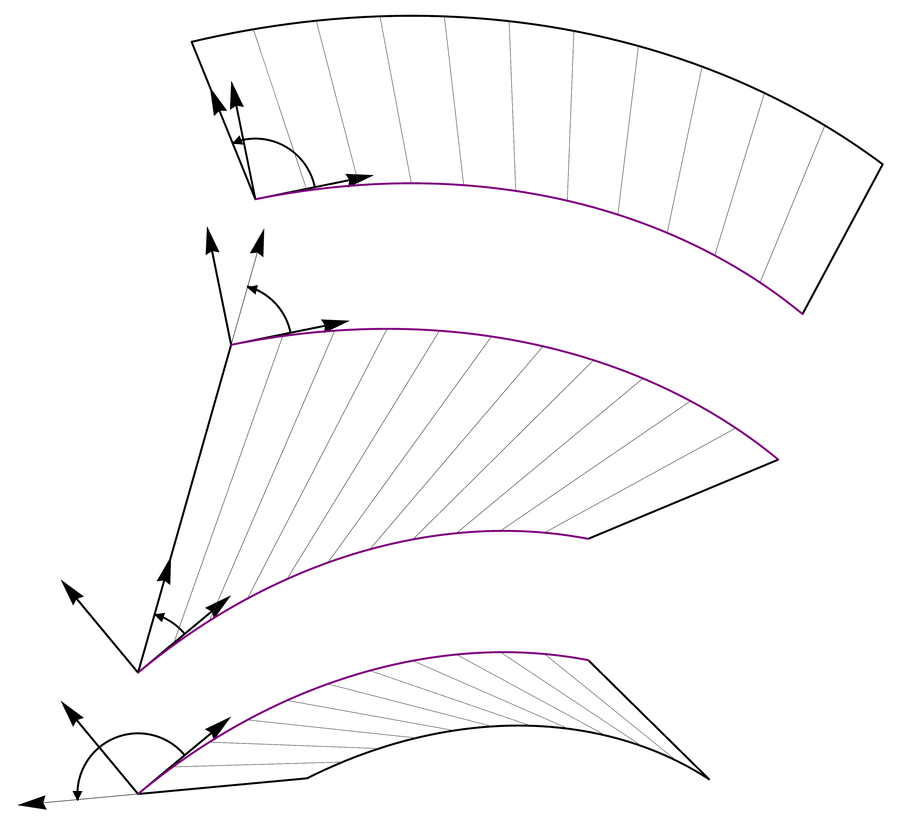}
        \put(10,11){$\theta_{1,R}(t)$}
        \put(68,17){$\vecx_{1}(s_1(t))$}
        \put(-3,4){$\vecr_{1,R}(t)$}

        \put(20,17){$\theta_{1,L}(t)$}
        \put(66,29){$\vecx_{1}(s_1(t))$}
        \put(20,30){$\vecr_{1,L}(t)$}

        \put(30,58){$\theta_{2,R}(t)$}
        \put(77,48){$\vecx_{2}(s_2(t))$}
        \put(30,63){$\vecr_{2,R}(t)$}

        \put(30,76){$\theta_{2,L}(t)$}
        \put(67,55){$\vecx_{2}(s_2(t))$}
        \put(11,80){$\vecr_{2,L}(t)$}

        \put(45,40){``central''}
        \put(40,12){``right''}
        \put(54,78){``left''}
    \end{overpic}
    \end{scriptsize}
    \begin{scriptsize}
    \begin{overpic}[width=0.49\textwidth,trim = 1cm 0.5cm 0.5cm 0cm, clip]{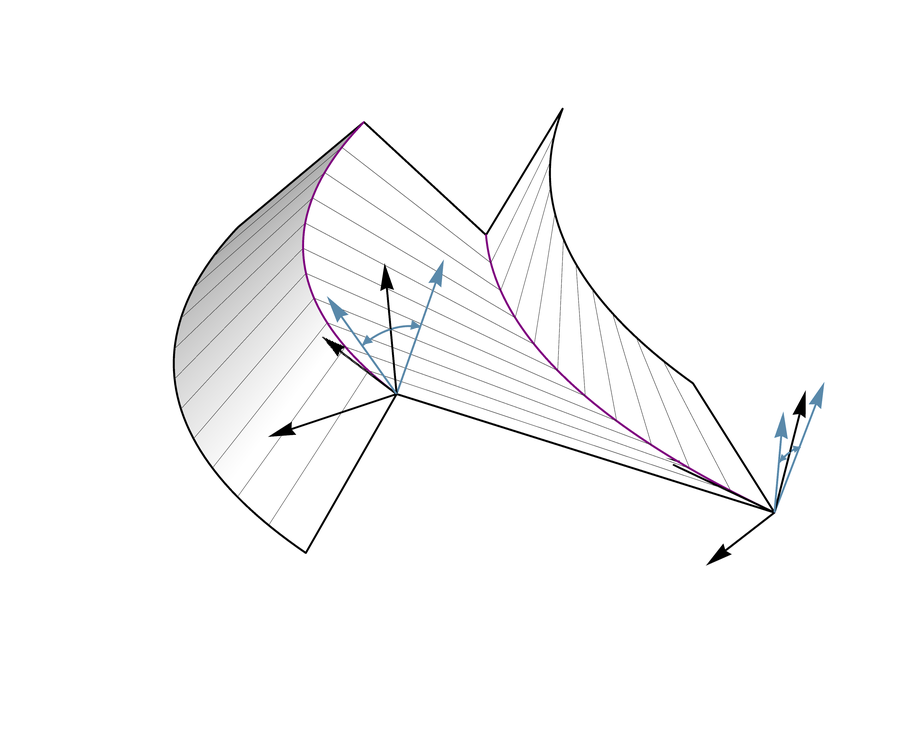}
        \put(90,21){$\vecX_1(t)$}
        \put(36,34){$\vecX_2(t)$}
        \put(75,14){$\vecT_1(t)$}
        \put(93,29){$\varphi_{1L}(t)$}
        \put(88,44){$\vecB_1(t)$}
        \put(42,45){$\varphi_{2R}(t)$}
        \put(23,42){$\varphi_{2L}(t)$}
        \put(32,61){$\vecB_2(t)$}
        \put(17,49){$\vecT_2(t)$}
        \put(78,33){$\varphi_{1R}(t)$}
        \end{overpic}%
       \end{scriptsize} %
  \caption{Illustration of the notation introduced in Section~\ref{sec:multiplecrease} (as in~\cite{mundilovaphd, mundilova2024planar}).}\label{fig:twocreases}
\end{figure}
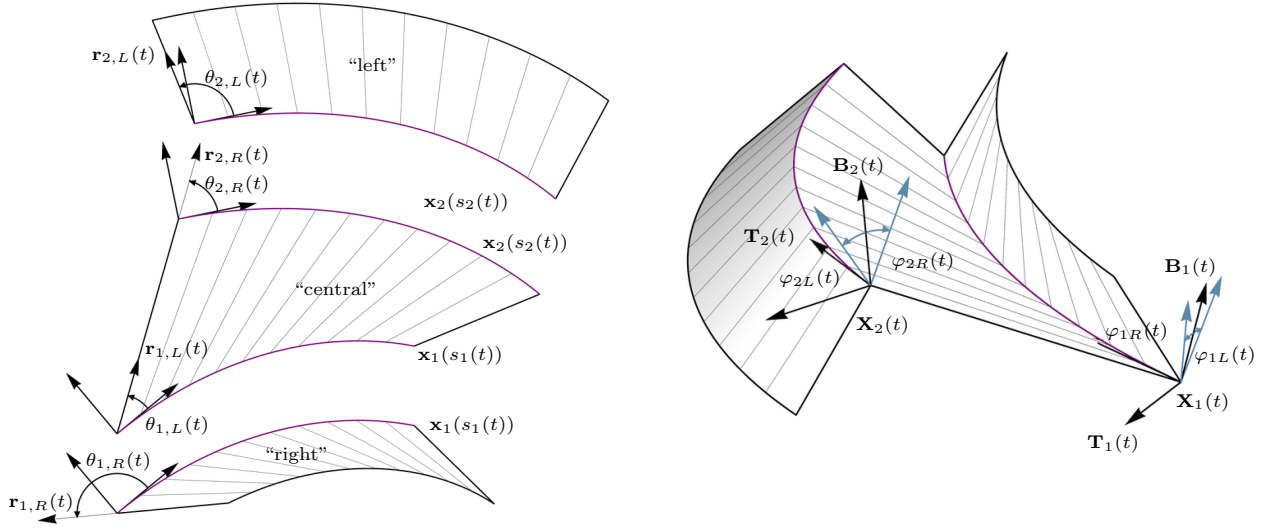

We continue by discussing the computation of folded states of regular crease-rule pattern $\mathcal{P}$, 
consisting of the $n$ folded crease curves $\vecX_i(t)$, determined by the corresponding  curvatures $K_i(t)$, torsions $\tau_i(t)$, and left and right inclination angles $\varphi_{iL}(t)$ and $\varphi_{iR}(t)$. 

Before we proceed, we conduct a similar simplification as in the single crease case. Specifically, as the curvatures and inclination angles are related by 
\begin{align}
k_{i}(t) = K_{i}(t) \cos \varphi_{iL}(t) &&\text{and}&& k_{i}(t) = K_{i}(t) \cos \varphi_{iR}(t),
\end{align}
we have that $\cos \varphi_{iL}(t) = \cos \varphi_{iR}(t)$.
Analogously to the single crease, we focus on the case where $\varphi_i(t) = \varphi_{iL}(t) = -\varphi_{iR}(t)$.
This simplification leaves us with $3n$ unknowns, namely the 3D curvatures $K_i(t)$, torsions $\tau_i(t)$, and inclination angles $\varphi_i(t)$. 

A valid folded state satisfies the following algebraic and differential constraints:

\begin{itemize}
    \item For every crease, the left and right patches must be developable, 
\begin{align}\label{eqn:dev3}
\varphi_{i}'(t) = \sigma_{ij} s_i'(t)\tau_{i}(t) + s_i'(t) k_{i}(t) \tan \varphi_{i}(t) \cot \theta_{ij}(t),
\end{align}
where $\sigma_{iL} = +1$ and $\sigma_{iR} = -1$. 
This results in $2n$ constraints. 

\item Additionally, for each crease, we relate the curvatures and inclination angles by 
\begin{align}\label{eqn:curv3}
k_{i}(t) = K_{i}(t) \cos \varphi_{i}(t), 
\end{align}
resulting in another $n$ constraints.

\item Finally, for every pair of overlapping surfaces, we require them to have the same ruling curvatures. This requirement leads to the following $n-1$ constraints:
\begin{equation}\label{eqn:common3}
s_i'(t) k_i(t) \tan \varphi_{i}(t) \frac{1}{\sin \theta_{i,L}(t)} = 
-s_{i+1}'(t) k_{i+1}(t) \tan \varphi_{i+1}(t) \frac{1}{\sin \theta_{i+1,R}(t)}.
\end{equation}
\end{itemize}

Note that a solution to these $4n-1$ constraints provides a sufficient condition for computing a folded state: the crease curves can be determined sequentially and are guaranteed to fit together by isometry (see Lemma~\ref{lem:rulcurv}). The boundary curves are obtained by propagating distances along the boundary developables from the outermost crease curves.

\begin{definition}\label{def:folded_state}
We say a regular crease-rule pattern $\mathcal{P}$ is said to admit a \emph{folded state} if there exist functions $K_i(t)$, $\tau_i(t)$, and $\varphi_i(t)$, such that 
\begin{itemize}
    \item Equation~\eqref{eqn:dev3} and Equation~\eqref{eqn:curv3} are satisfied for every crease, and 
    \item Equation~\eqref{eqn:common3} is satisfied for every interior patch.
\end{itemize} 
A folded state is \emph{non-trivial}, if no inclination angle is identically zero.

We say a crease-rule pattern has a \emph{rigid-ruling folding motion}, if it has a continuous family of non-trivial folded states. 
\end{definition}

Ultimately, for a regular crease-rule pattern with 
$n$ creases, the system involves $4n-1$ constraints imposed on $3n$ variables. Thus, for $n>1$, it is typically overconstrained. This is due to each pair of adjacent patches acting as a one-degree-of-freedom mechanism, which may fail to produce a congruent configuration of the shared surface. Strategies for introducing additional degrees of freedom into this system are discussed in~\cite[Part I]{mundilovaphd}.

\paragraph{Remark on vanishing curvatures.}
Note that the definition of a crease-rule pattern does not explicitly forbid crease curves containing points or segments of vanishing curvature. Intuitively, in regions where $k_i(t)$, a folded configuration corresponds to two flat regions joined along the crease. To avoid case distinctions, the following discussion focuses only on crease patterns whose crease curvatures vanish only at isolated points.

For Equation~\eqref{eqn:common3} to be satisfied at such parameter values without introducing singularities, the curvatures of creases connected by a ruling should also vanish. Accordingly, we restrict our attention to patterns whose inflection points are connected by rulings:

\begin{definition}
Let $\mathcal{P}$ be a regular crease-rule pattern with at least one crease, and let $T_0 \subseteq T$ denote the set of isolated parameter values for which $k_1(t) = 0$. We call $\mathcal{P}$ a \emph{candidate crease-rule pattern}, if for all other creases in the pattern, $k_1(t) = 0$ implies $k_i(t) = 0$. 
\end{definition}

While we conjecture that crease-rule patterns which are not candidate patterns do not admit a rigid-ruling motion, we omit the corresponding technical details here. Our main result in Section~\ref{sec:mainresult} provides two necessary and sufficient conditions for a candidate crease-rule pattern to admit a rigid-ruling folding motion.

\subsection{Property-preserving operations on crease-rule patterns}\label{sec:parallel}

To conclude the preliminary discussions, we describe two operations on crease-rule patterns that involve parallelism. Both operations preserve the existence of folded states, rigid-ruling foldability of the patterns, and the types of curved creases (planar or constant fold-angle). Although both operations naturally extend to general semi-discrete conjugate nets, see~\cite{mundilovaphd}, in this paper we focus exclusively on the globally developable case.

\subsubsection{Preliminaries}\label{sec:comb_prelim}

We begin by reviewing the properties and computational methods of curves with parallel frames.

\paragraph{Curves with parallel frames.} Recall that curves with parallel frames (and same orientation of vectors) at corresponding parameter values satisfy the same Frenet–Serret formulas~\cite{hatzidakis1902om,salkowski1909transformation}.  Consequently, the quantities of a 3D curve $\vecX(t)$ with curvature $K(t)$, torsion $\tau(t)$, and arc-length $s(t)$ relate to the curvature $\tilde K(t)$, torsion $\tilde\tau(t)$, and arc-length $\tilde s(t)$ of a curve with parallel frame by 
\begin{align*}
s'(t)K(t) = \tilde s'(t) \tilde K(t) && \text{ and } && s'(t) \tau(t) = \tilde s'(t) \tilde \tau(t),
\end{align*}
and the curves are related through integration of the tangent vectors with different parametrization speeds, that is, 
$\vecX'(t) = s'(t) \vecT(t)$ and $\tilde{\vecX}'(t) = \tilde{s}'(t) \vecT(t)$, where $\vecT(t)$ is the common tangent vector. Equivalently, the curves relate by $s'(t) \tilde\vecX'(t) = \tilde s'(t) \vecX'(t)$.

Similarly, in the two-dimensional case, the curvatures $k(t)$ and $\tilde k(t)$ of two parallel curves $\vecx(t)$ and $\tilde{\vecx}(t)$ with arc-lengths $\tilde{s}(t)$ and $\tilde{\vecx}(t)$ relate by
$$
s'(t)k(t) = \tilde s'(t) \tilde k(t),
$$
and we have that $\vecx'(t) = s'(t) \vect(t)$ and $\tilde{\vecx}'(t) = \tilde{s}'(t) \vect(t)$, where $\vect(t)$ is the common tangent vector.

\paragraph{Curves on developable patches with parallel frames.}  Among the many curves with parallel frames, the following crease-rule pattern modifications make use of finding curves with parallel frames on a given developable surface. In what follows, we briefly discuss the corresponding computation. 

Let $\vecs(t,u) = \vecx(t) + u \vecr(t)$ be a 2D patch and $\vecx_t(t)$ a 2D curve. Our goal is to locate a curve $\tilde\vecx(t)$ on $\vecs(t,u)$ such that its tangents at corresponding parameters $t$ are parallel to the tangents of $\vecx_t(t)$. While the following computation requires $|\vecx_t'(t) \times \vecr(t)| \neq 0$ for all $t\in T$, we do not assume that the curve $\vecx_t(t)$ lies on $\vecs(t,u)$. 

To compute the unknown curve $\tilde\vecx(t)$, we assume that it can be parametrized by $\tilde{\vecx}(t) = \vecx(t) + l(t) \vecr(t)$, where $l(t)$ is an initially unknown function. Let $\vecn_t(t)$ denote the normal vector of the curve $\vecx_t(t)$. Requiring the derivative $\tilde{\vecx}'(t)$ to be orthogonal to $\vecn_t(t)$, that is, $\tilde{\vecx}'(t) \cdot \vecn_t(t) =0$, implies
$$
l'(t) = 
- \frac{\vecn_t(t) \cdot \vecr'(t)}{\vecn_t(t) \cdot \vecr(t) } l(t)  - \frac{\vecn_t(t) \cdot \vecx'(t)}{\vecn_t(t) \cdot \vecr(t)}.
$$
This condition is an explicit first-order differential equation for the length function $l(t)$, solved by 
$$
l(t) = e^{\int_0^t a(\xi)} d\xi \left( l(0) + \int_0^t e^{-\int_0^{\eta} a(\xi)d\xi} b(\eta) d\eta\right), \quad \text{where} \quad a(t) = - \frac{\vecn_t(t) \cdot \vecr'(t)}{\vecn_t(t) \cdot \vecr(t) } \text{ and } b(t) = - \frac{\vecn_t(t) \cdot \vecx'(t)}{\vecn_t(t) \cdot \vecr(t)}.
$$
Given a reasonable combination of developable patch, curve, and initial value $l(0)$, we therefore expect to find a local solution for $l(t)$.

\subsubsection{Combescure-transformations of crease-rule patterns}\label{sec:parcurves}

Recall that two smooth or discrete conjugate nets are said to be related by a Combescure transformation if, at each point, the corresponding partial derivatives or edges are parallel~\cite{combescure1867determinants,eisenhart1923transformations,kilianSmoothDiscreteConeNets2023}.

In this section, we describe how to compute a semi-discrete Combescure transformation $\tilde{\mathcal{P}}$ of a crease-rule pattern $\mathcal{P}$ and highlight key properties of the transformed pattern. This operation is particularly useful for simplifying casework, as used in~\cite{mundilova2024planar} or later in Section~\ref{sec:mainresult}.

In the following, suppose we are given a crease-rule pattern $\mathcal{P}$ containing $n$ creases $\vecx_i(t)$ and their incident rulings, a continuous function $p_0(t) > 0$, and $n+1$ positive initial lengths $(l_0(0), \dots, l_{n}(0))$.

To initialize the construction, we use the function $p_0(t)$ to determine the boundary curve $\tilde{\vecx}_0(t)$ of $\tilde{\mathcal{P}}$ parallel to $\vecx_0(t)$. Specifically, we set  
    \begin{align*}
\tilde s_0'(t) = p_0(t) s_0'(t) 
\end{align*}
and obtain the first curve $\tilde\vecx_0(t)$ from integrating $\tilde \vecx_0'(t) = p_0(t) \vecx_0'(t)$, or equivalently $\tilde\vecx_0'(t) = \tilde s_0'(t) \vect_0(t)$, where $\vect_0(t)$ is the tangent vector of $\vecx_0(t)$.
Furthermore, we define $\tilde{\vecs}_{0,L} = \tilde \vecx_0(t) + u \vecr_{0,L}(t)$, where $\vecr_{0,L}(t)$ are the rulings of the patch between $\vecx_0(t)$ and $\vecx_1(t)$. 

While there is a crease $\vecx_i(t)$ to the left of the last integrated curve corresponding to curve $\vecx_{i-1}(t)$, for a specified initial distance $l_{i-1}(0)$, we find a curve $\tilde \vecx_i(t)$ on $\tilde \vecs_{i-1,L}(t,u)$ such that its tangents are parallel to $\vecx_i(t)$, see Section~\ref{sec:comb_prelim}. 
Upon a successful computation, we again set $\tilde\vecs_{i,L}(t,u) = \tilde\vecx_i(t) + u \vecr_{i,L}(t)$.

For reasonable input, this procedure describes the sequential construction of a Combescure-transformed crease-rule pattern $\tilde{\mathcal{P}}$.
We make note of the following property that follows directly from the definition of planar and constant fold-angle creases and parallelism:

\begin{lemma}
If $\vecx_i(t)$ is a planar or a constant fold-angle crease, then $\tilde{\vecx}_{i}(t)$ has the same property.
\end{lemma}

Next, we show that a Combescure relationship between two crease patterns induces a corresponding Combescure relationship between their folded states: 

\begin{lemma}
If crease-rule pattern 
$\mathcal{P}$ has a folded state $\mathcal{F}$, then its Combescure transform $\tilde{\mathcal{P}}$ also has a folded state $\tilde{\mathcal{F}}$. 
\end{lemma}

\begin{proof}
We prove this statement by constructing a valid folded state of $\tilde{\mathcal{P}}$ from a folded state $\mathcal{P}$.

First, we obtain $\tilde\vecX_0(t)$ by integrating $\tilde \vecX_0'(t) = p_0(t) \vecX_0'(t)$, and define $\tilde \vecS_{0,L}(t,u) = \tilde{\vecX}_0(t) + u \vecR_{0,L}(t)$ using the ruling direction $\vecR_{0,L}(t)$ of the patch between $\vecX_0(t)$ and $\vecX_1(t)$ in the folded state of $\mathcal{P}$.
While there is a curve to the left of a previously computed curve $\tilde\vecX_{i-1}(t)$, we locate $\tilde\vecX_i(t)$ as a curve on $\tilde\vecS_{i-1,L}(t,u)$, and set $\tilde\vecS_{i,L}(t,u) = \tilde\vecX_i(t) + u \vecR_{i,L}(t)$.

Note that due to isometry, the angles between rulings and tangents are unchanged, and the curves $\tilde{\vecX}_i(t)$ and $\vecX_i(t)$ have parallel tangents, and consequently frames. 
It follows that there exist positive continuous functions $p_i(t)$ that relate the quantities associated with the curves $\tilde\vecX_i(t)$ and $\vecX_i(t)$ by 
$$
\tilde s_i(t) = p_i(t) s_i(t), \quad
\tilde K_i(t) = p_i(t) K_i(t), \quad \text{ and } \quad
\tilde \tau_i(t) = p_i(t) \tau_i(t). 
$$
Additionally, again due to isometry, we have that 
$$
\tilde k_i(t) = p_i(t) k_i(t).
$$

We conclude that the constructed semi-discrete structure is a valid folded state of $\tilde{\mathcal{P}}$, since the quantities $K_i(t)$, $\tau_i(t)$, and  $\varphi_i(t)$ satisfying the conditions in Definition~\ref{def:folded_state}, imply that the corresponding quantities $\tilde K_i(t)$, $\tilde\tau_i(t)$, and $\varphi_i(t)$ also satisfy them. 

Finally, we note that, by construction, both the curve tangents and rulings of $\mathcal{F}$ and $\tilde{\mathcal{F}}$ are parallel, and therefore the folded states relate by a Combescure transformation.
\end{proof}

Finally, we conclude that the folding properties of the crease-rule patterns $\mathcal{P}$ and $\tilde{\mathcal{P}}$ are the same:

\begin{corollary}
    The crease pattern $\tilde{\mathcal{P}}$ admits a folded state if and only if $\mathcal{P}$ does. Consequently, $\tilde{\mathcal{P}}$ admits a rigid-ruling folding motion if and only if $\mathcal{P}$ admits one.
\end{corollary}

\subsubsection{Adding parallel pleats to crease-rule patterns}\label{sec:parrul}

\begin{figure}
    \centering
\includegraphics[width=0.19\textwidth,trim = 2.2cm 1.7cm 0.9cm 1cm, clip]{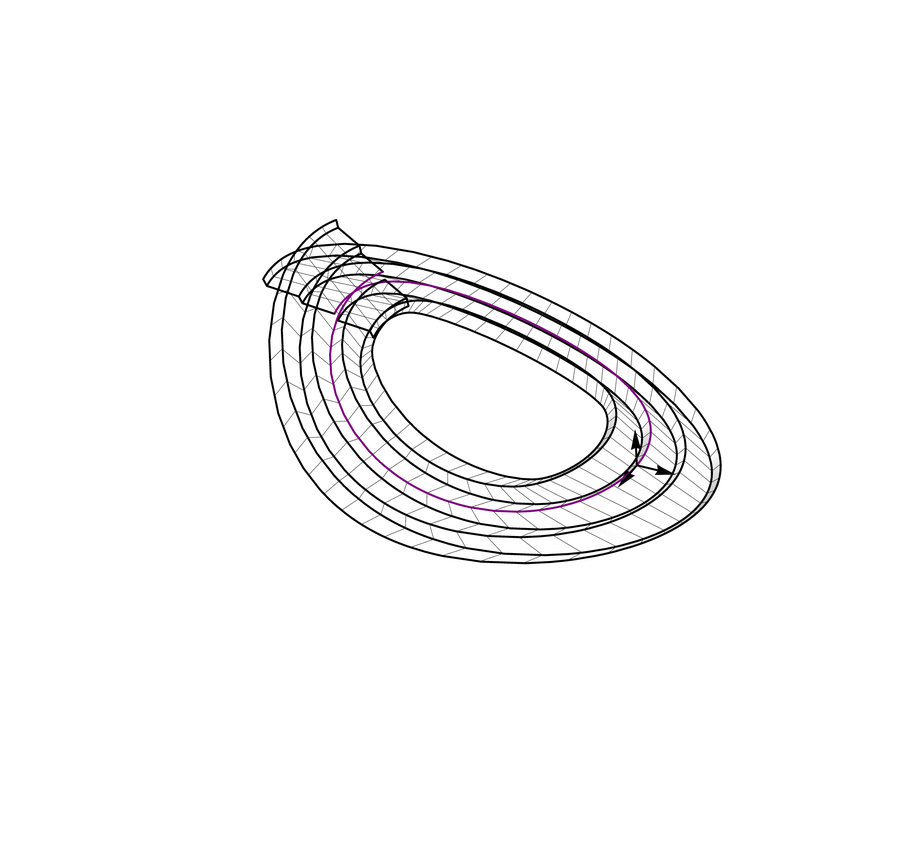}
\includegraphics[width=0.19\textwidth,trim = 2.2cm 1.7cm 0.9cm 1cm, clip]{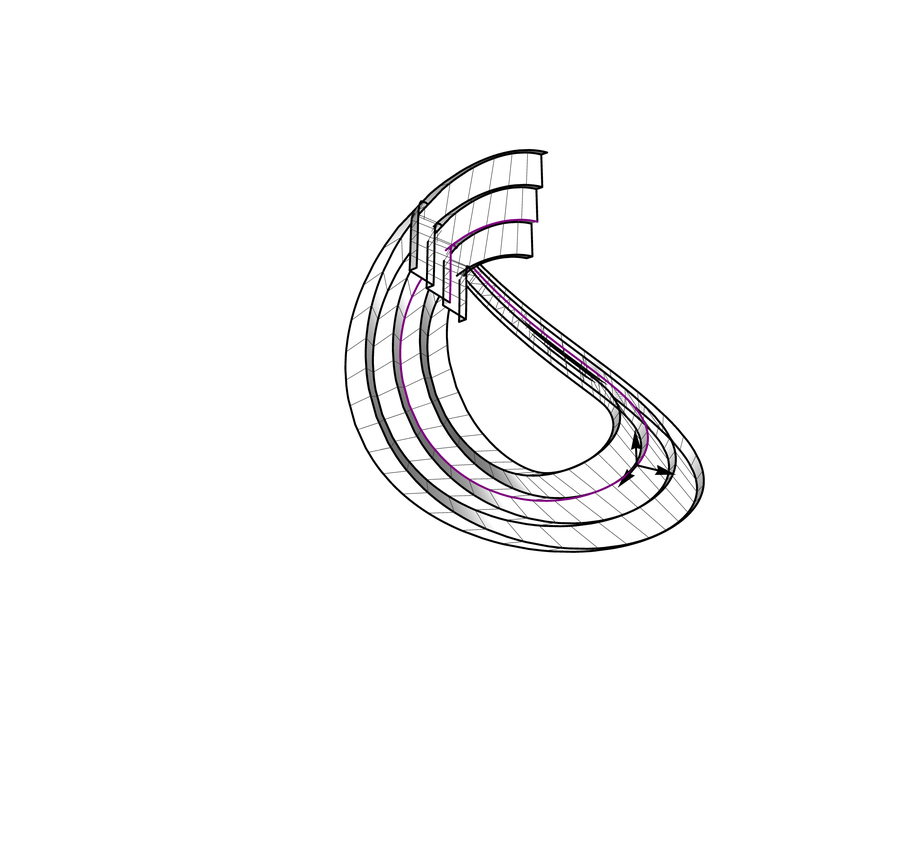}
\includegraphics[width=0.19\textwidth,trim = 2.2cm 1.7cm 0.9cm 1cm, clip]{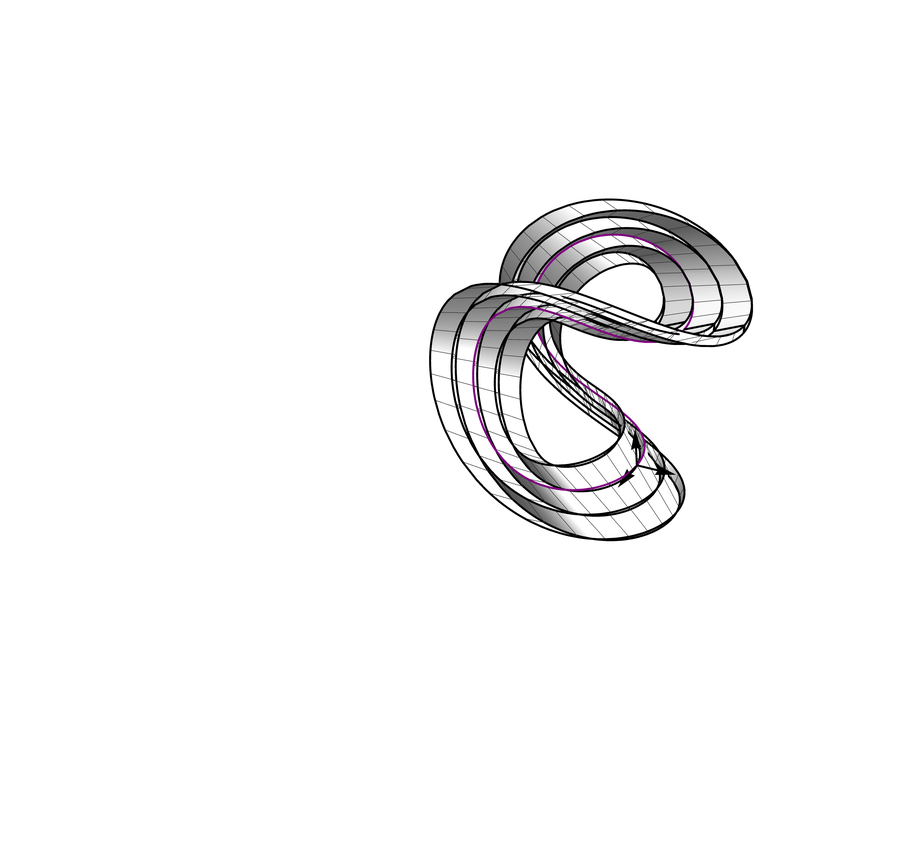}
\includegraphics[width=0.19\textwidth,trim = 2.2cm 1.7cm 0.9cm 1cm, clip]{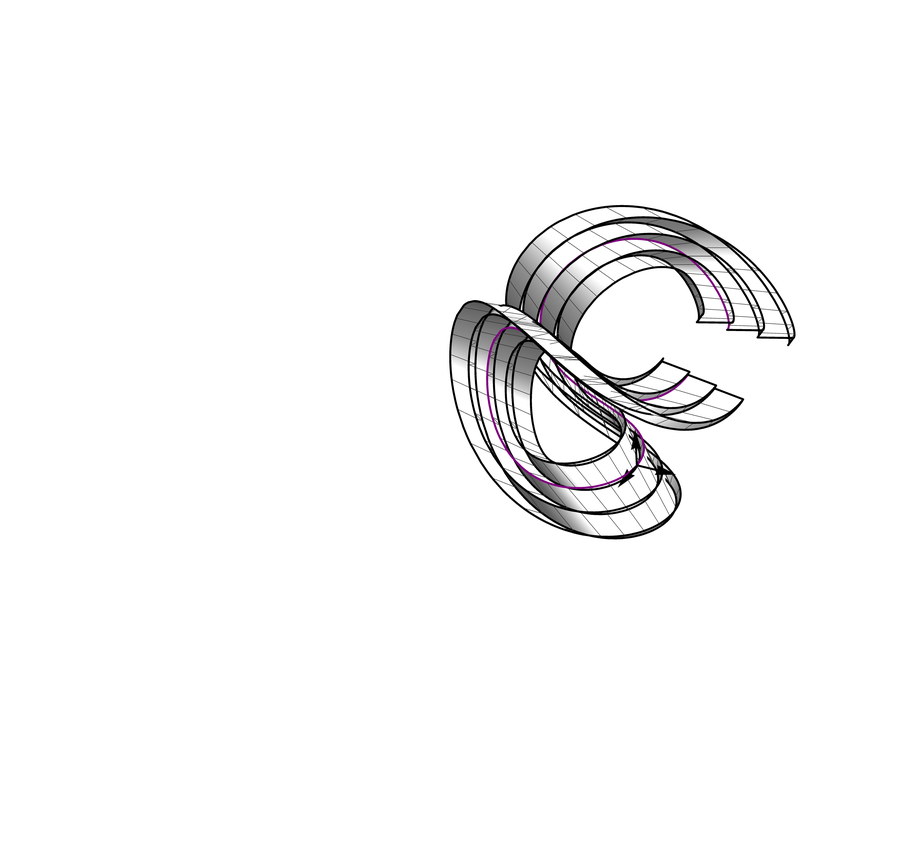}
\includegraphics[width=0.19\textwidth,trim = 2.2cm 1.7cm 0.9cm 1cm, clip]{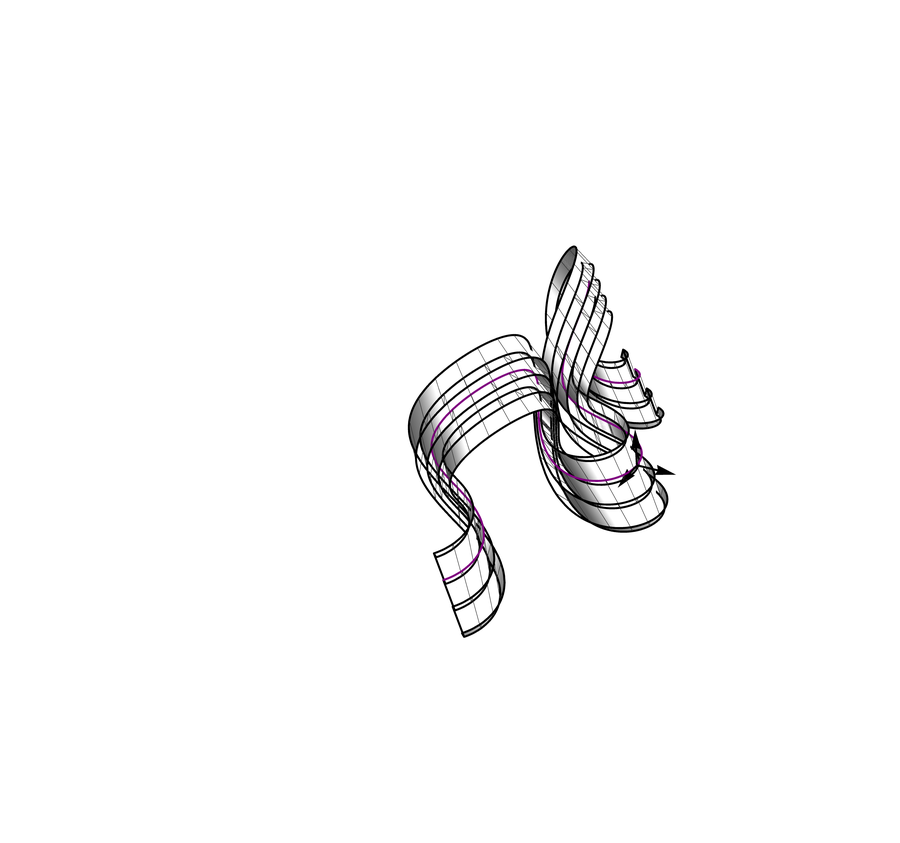}
\caption{Illustration of a rigid-ruling folding motion of a crease-rule pattern with tangent-parallel creases~\cite{mundilovaphd}.}\label{fig:parRul}
\end{figure}

Next, we review a Combescure-inspired method that uses parallelism to add creases to a crease-rule pattern~\cite{tachi2011designing, jiang2020}, while again preserving its folding properties, see Figure~\ref{fig:parRul}.

In the following, let $\mathcal{P}$ be a crease-rule pattern and $l_n(0)$ and $l_{n+1}(0)$ two initial values. We discuss the generation of a crease pattern $\tilde{\mathcal{P}}$ containing the curves $\vecx_0(t), \dots, \vecx_n(t)$ of $\mathcal{P}$ with one additional crease, $\tilde\vecx_{n+1}(t)$ together with a new left boundary $\tilde\vecx_{n+2}(t)$. 

Specifically,  using the method described in Section~\ref{sec:comb_prelim} and the first initial value $l_n(0)$, we identify a curve $\tilde\vecx_{n+1}(t)$ on $\vecs_{n,L}(t,u)$ whose tangents are parallel to the tangents of $\vecx_n(t)$ at points connected by a ruling, and set $\tilde\vecs_{n+1,L}(t,u) = \tilde\vecx_{n+1}(t) + u \vecr_{n-1,L}(t)$. Additionally, using the second initial value $l_{n+1}(0)$, we identify the left-most boundary as a curve parallel to $\vecx_{n+1}(t)$ on $\tilde\vecs_{n+1,L}(t,u)$.

 Note that this construction ensures not only that the curves $\vecx_n(t)$ and $\tilde \vecx_{n+1}(t)$ are parallel, but also that the rulings incident to these curves are parallel. Specifically, $\vecr_{n+1,L}(t) = \vecr_{n,R}(t)$ and $\vecr_{n+1,R}(t) = \vecr_{n,L}(t)$. A direct consequence is the following: 

 \begin{lemma}
If the crease $\vecx_n(t)$ is a planar or a constant fold-angle crease, so is $\tilde\vecx_{n+1}(t)$.
\end{lemma}

Next, with similar arguments to before, we show that: 

\begin{lemma}
If the crease-rule pattern $\mathcal{P}$ has a folded state, then $\tilde{\mathcal{P}}$ also has a folded state. 
Consequently, the existence of a rigid-ruling folding motion of $\mathcal{P}$ implies the existence of a rigid-ruling folding motion of $\tilde{\mathcal{P}}$.
\end{lemma}

\begin{proof}
Similar to the previous section, we can use isometry to construct a 3D configuration corresponding to the crease-rule pattern $\tilde{\mathcal{P}}$. 

To show that this is a folded state, we confirm that the equations in Definition~\ref{def:folded_state} are satisfied. Specifically, it follows from the definition of $\tilde\vecx_{n+1}(t)$, that the curves $\vecx_n(t)$ and $\vecx_{n+1}(t)$ parametrization speeds $s_{n}'(t)$ and $\tilde s_{n+1}'(t)$, as well as their curvatures $k_n(t)$ and $\tilde k_{n+1}(t)$, differ by a factor $p(t)>0$, namely,
\begin{align*}
    s_{n+1}'(t) = p(t) s_{n}'(t) && \text{and} && k_{n+1}(t) = \frac{1}{p(t)} k_n(t), 
\end{align*}
and similarly, the quantities of the 3D curves relate by 
\begin{align*}
    \tilde K_{n+1}(t) = p(t) K_{n}(t) && \text{and} && \tilde \tau_{n+1}(t) = \frac{1}{p(t)} \tau_n(t). 
\end{align*}
Additionally, we have that $\theta_{n+1,L}(t) = \theta_{n,R}(t)$ and $\theta_{n+1,R}(t) = \theta_{n,L}(t)$.

We observe that if quantities associated with $\mathcal{P}$ satisfy Definition~\ref{def:folded_state}, so will the quantities associated with $\tilde{\mathcal{P}}$ with $\tilde\varphi_{i+1}(t) = \varphi_i(t)$.
\end{proof}
\section{Rigid-Ruling Folding Conditions for Crease-Rule Patterns}\label{sec:mainresult}

Given the over-constrained nature of the system discussed in Section~\ref{sec:multiplecrease}, a generic crease-rule pattern does not admit a folded state, and the existence of a rigid-ruling folding motion constitutes a special case. In this section, we present our novel results on the rigid-ruling foldability of regular crease-rule patterns.

\subsection{Local conditions for rigid-ruling foldability} 

 It is established that a regular PQ-mesh (not necessarily developable) is rigidly foldable if and only if every $(3 \times 3)$-submesh is rigidly foldable~\cite{schief2008integrability}. 
Similarly, for a crease-rule pattern with $n>$ creases, the question of whether it admits a rigid-ruling folding motion can be reduced to a sequence of smaller problems:

\begin{lemma}\label{lem:local}
In a regular crease-rule pattern, if each pair of adjacent creases and their incident surfaces can undergo a rigid-ruling folding motion, then the entire structure will likewise be capable of such a motion. 
\end{lemma}

\begin{proof}
This deduction follows from the property that each pair of adjacent surfaces forms a one-degree-of-freedom mechanism with a family of folded states near the planar configuration. Hence, the ability to combine these mechanisms pairwise is sufficient to guarantee that the entire structure can be assembled.
\end{proof}

\subsection{Rigid-ruling folding compatibility of two creases}\label{sec:neccessary}

For a regular crease-rule pattern, Lemma~\ref{lem:local} implies that it suffices to examine pairs of rigid-ruling foldable creases. In this section, we present our main result, which extends previous work on the rigid-ruling compatibility of two planar creases~\cite{mundilova2024planar}, and establishes two conditions under which a crease-rule pattern $\mathcal{P}$ with two creases is rigidly foldable.
Throughout this section, we continue using the notation introduced in Section~\ref{sec:primer}.

\subsubsection{Preparations}

Recall from Section~\ref{sec:singlecrease}, that the folded states of a pair of patches joined along a curved crease $\vecx_i(t)$ are determined by their inclination angles 
\begin{equation}\label{eqn:phiI}
\varphi_i(t) = \arcsin\left(c_i e^{I_i(t)}\right) 
\end{equation}
for
\begin{equation}\label{eqn:defI}
I_i(t) = \frac{1}{2}\int^t_0 s_i'(\tilde t) k_i(\tilde t) \left(\cot \theta_{iL}(\tilde t) + \cot \theta_{iR}(\tilde t) \right)d\tilde t, 
\end{equation}
where $c_1$ and $c_2$ are two constants satisfying 
$$
\left|c_i \right|\leq c_{i,\max} := \min_{t\in T} e^{-I_i(t)}.
$$

We substitute the representation of the inclination angle from Equation~\eqref{eqn:phiI} into the compatibility condition between two patches with the same ruling curvature, given in Equation~\eqref{eqn:common3}. This yields
\begin{equation}\label{eqn:samerulcurve}
\frac{s_1'(t) k_1(t)}{\sin\theta_{1L}(t)} \tan\left(\arcsin\left(c_1 e^{I_1(t)}\right)\right)= - \frac{s_2'(t) k_2(t)}{\sin\theta_{2R}(t)} \tan\left(\arcsin\left(c_2 e^{I_2(t)}\right)\right).
\end{equation}
Using the trigonometric equality $\tan(\arcsin(x)) = x(1 - x^2)^{-\frac{1}{2}}$, we rewrite this condition as 
\begin{equation}\label{eqn:comp}
F_1(t) \frac{c_1 e^{I_1(t)}}{\sqrt{1 - c_1^2 e^{2 I_1(t)}}} =  - F_2(t)\frac{c_2 e^{I_2(t)}}{\sqrt{1 - c_2^2 e^{2 I_2(t)}}}
\end{equation}
where
\begin{equation*}
    F_1(t) := \frac{s_1'(t)k_1(t)}{\sin \theta_{1L}(t)}
    \quad 
    \text{and}
    \quad 
    F_2(t) := \frac{s_2'(t) k_2(t)}{\sin \theta_{2R}(t)}.
\end{equation*}

For a rigid-ruling folding motion to exist, we require that this equation is satisfied for a one-parameter family of pairs of initial values $(c_1, c_2)$, starting at the flat-configuration with $(c_1, c_2) = (0,0)$. In the following, we assume that $c_1, c_2 \neq 0$.

\subsubsection{Compatibility during folding motion}

With these preparations in place, we present the following theorem.

\begin{theorem}[Compatibility of creases]\label{thm:main}
A candidate crease-rule pattern has a rigid-ruling folding motion if and only if for all non-zero curvature parameters $t\in T \backslash T_0$, 
\begin{equation}\label{eqn:condA}
 \frac{F_1'(t)}{F_1(t)} - \frac{F_2'(t)}{F_2(t)} + I_1'(t) - I_2'(t) = 0
\end{equation}
and 
\begin{equation}\label{eqn:condB}
F_2(t)^2 I_1'(t) = F_1(t)^2 I_2'(t).
\end{equation}
are satisfied.
\end{theorem}

\begin{proof}
Our proof consists of three steps. In the first step, we focus on parameter values $t \in T\backslash T_0$ and  derive two constraints that ensure Equation~\eqref{eqn:comp} is satisfied for a one-parameter family of values 
$c_1$ and $c_2$. In the second step, we show that these two constraints imply a smooth connection between the initial values $c_1$ and $c_2$. Finally, we argue the existence of a rigid-ruling folding motion by considering all parameter values $t\in T$.

\paragraph{Step 1 (Derivation of constraints).} In the following, we assume that $t\in T\backslash T_0$, and thus $F_i(t) \neq 0$.
Solving Equation~\eqref{eqn:comp} for the square of $c_2$ yields
\begin{align*}
c_2^2 &= \frac{A(t) c_1^2}{B(t) c_1^2 + C(t)},
\end{align*}
where
\begin{align*}
    A(t) &= F_1(t)^2 e^{2 I_1(t)}, 
    & B(t) &= \left(F_1(t)^2 - F_2(t)^2\right)e^{2 I_1(t)} e^{2 I_2(t)}, 
    & C(t) &= F_2(t)^2e^{2I_2(t)}.
\end{align*}
Note that the denominator cannot vanish for all $t$, which is equivalent to $B(t) = C(t) = 0$, since our assumptions imply that $C(t) \neq 0$.

The value of $c_2$ is constant in $t$, if and only if its differentiation w.r.t.\ parameter $t$ is zero, that is,
\begin{align*}
0 & = \frac{d}{dt} c_2^2 = 
\frac{
\left(A'(t) B(t) - A(t) B'(t)\right)c_1^2 +A'(t)C(t)
-A(t) C'(t)  }{( B(t)c_1^2 + C(t))^2}c_1^2.
\end{align*}
This equation must be satisfied for all suitable values of $c_1$. This is the case if and only if  the coefficients of $c_1^4$ and $c_1^2$ in the numerator vanish, leading to the following two constraints:
\begin{align}\label{eqn:eqn3}
0 &= A'(t) B(t)  - A(t) B'(t) & \text{and} && 
0 &= A'(t) C(t) - A(t) C'(t).
\end{align}

First, we simplify Equation~\eqref{eqn:eqn3} (right), which results in Equation~\eqref{eqn:condA}.
Next, we consider the equation which results from adding $e^{2I_1(t)}$ times the right equation of Equation~\eqref{eqn:eqn3} to its left one:
\begin{equation}
A'(t) \left(B(t) + e^{2I_1(t)} C(t)\right)
- A(t) \left(B'(t) + e^{2I_1(t)}C'(t)\right) = 0. 
\end{equation}
Simplifications of this equation result in Equation~\eqref{eqn:condB}.

Note that it follows from Equation~\eqref{eqn:condB} that $I_1'(t) = 0$ if any only if $I_2'(t) =0$. Since constant fold-angle creases $\vecx_i(t)$ are characterized by $I_i'(t) = 0$, it follows that if one crease has a constant fold-angle, for the crease-rule pattern to be rigid-ruling foldable, then the other must as well.

\paragraph{Step 2 (Smooth connection between initial values).} 
Next, we show that $c_2$ smoothly depends on $c_1$ and that there exists a motion that starts in a configuration where both values are zero. We continue assuming that $t\in T\backslash T_0$. 

\begin{itemize}
\item \emph{Constant fold-angle creases:} First, we consider the case where both creases have constant fold-angle, that is,  $I_1(t) = I_2(t) = 0$ for all $t\in T$. Integrating Equation~\eqref{eqn:condA}, we obtain 
$$
\frac{F_1(t)}{F_2(t)} = c_3,
$$
where $c_3 > 0$ is a constant determined by the geometry of the crease-rule pattern. 

To establish a relationship between $c_1$ and $c_2$, we consider Equation~\eqref{eqn:comp}, which simplifies to 
\begin{equation}\label{eqn:conn_const}
c_3 \frac{c_1}{\sqrt{1-c_1^2}} = - \frac{c_2}{\sqrt{1-c_2^2}}. 
\end{equation}
It follows that 
\begin{equation}\label{eqn:c2_const}
c_2^2 = \frac{c_1^2 c_3^2}{1-c_1^2(1-c_3^2)}. 
\end{equation}
Consequently, the constant $c_2$ is only real-valued if $1-c_1^2(1-c_3^2)>0$. This is the case either when $c_3^2 > 1$, or $c_3^2 < 1$ and $c_1^2 < 1/(1-c_3^2)$. 

In these two cases, back-substitution of the two solutions for $c_2$ of Equation~\eqref{eqn:c2_const} in Equation~\eqref{eqn:conn_const}, results in
$$
c_2 = - \frac{c_1c_3}{\sqrt{1-c_1^2(1-c_3^2)}}, 
$$
which shows a smooth relationship between the constants and thus implies the existence of a smooth folding motion that starts from the flat state. 

These consideration establish a maximal attainable fold-angle based on the crease-rule pattern.
Namely, if $c_3^2 < 1$, we have that $c_1 \in (-c_{1,\max},c_{1,\max})$ where $c_{1,\max} := 1 /\sqrt{1-c_3^2}$, otherwise, $c_1$ is not not constrained and a folding motion exists for all $c_1$.

\item \emph{Not constant fold-angle creases:} Next, we consider the case where $I_1'(t) = I_2'(t) = 0$ only for isolated values. 

Integrating Equation~\eqref{eqn:condA} yields 
\begin{equation*}
0 = I_1(t) - I_2(t) + \ln F_1(t) - \ln F_2(t) + \operatorname{const}, 
\end{equation*}
which implies 
\begin{equation}\label{eqn:condAint}
\frac{F_1(t) e^{I_1(t)}}{F_2(t) e^{I_2(t)}} = c_3,
\end{equation}
where $c_3 > 0$ is a constant that depends on the given crease-rule pattern. 

In the following, we assume $t\in T$ with $I_i'(t) \neq 0$, and consider the case where $I_i'(t) = 0$ futher below.
Combining Equation~\eqref{eqn:condAint} with Equation~\eqref{eqn:condB} results in
\begin{equation}\label{eqn:condBint}
I_1'(t)e^{2 I_1(t)} = c_3^2  I_2'(t)e^{2 I_2(t)},
\end{equation}
and another integration step yields
\begin{equation}\label{eqn:c4}
e^{2I_1(t)} = c_3^2 e^{2I_2(t)} + c_4,
\end{equation}
where $c_4$ is a constant that again is determined by the crease pattern, specifically, 
$$
c_4 = e^{2 I_1(t)}\left(1 - \frac{F_1(t)^2}{F_2(t)^2}\right).
$$

To find the connection between values of $c_1$ and $c_2$, we substitute the derived relations into the compatibility condition. Specifically, inserting Equation~\eqref{eqn:condAint} in Equation~\eqref{eqn:comp} simplifies to 
$$
c_3 = - \frac{c_2}{c_1}\frac{\sqrt{1-c_1^2(c_3^2 e^{2 I_2(t)} + c_4)}}{\sqrt{1-c_2^2 e^{2 I_2(t)}}},
$$
and solving for $c_2^2$ yields
\begin{align}\label{eqn:constconn}
c_2^2 = \frac{c_1^2 c_3^2}{1- c_1^2 c_4}.
\end{align} 
Again, $c_2$ is real-valued, if $c_4< 0$, or $c_4 > 0$ and $c_1^2 < 1/c_4$. 
If this is the case, we have that 
$$
c_2 = - \frac{c_1c_3}{\sqrt{1-c_1^2 c_4}} 
$$ 
is a smooth relationship between the constants $c_1$ and $c_2$, and hence there exists a smooth folding motion.

We extend these considerations for isolated values where $I_1(t) = I_2(t) = 0$, and note that Equation~\eqref{eqn:c4} implies that $c_4 = 1 - c_3^2$, which establishes the connection to the constant fold-angle case. 
\end{itemize}

This concludes our argumentation that the two constraints are also sufficient for the existence of a rigid-ruling motion for values of $t$ that do not correspond to inflection points of the creases in the pattern.

\paragraph{Step 3 (Parameter values with zero curvature).} Finally, we extend our considerations to all parameter-values $t\in T$. It follows from Equation~\eqref{eqn:comp} that at parameter values where $k_1(t) = k_2(t) = 0$, the compatibility condition is trivially satisfied and does not pose a constraint on the connection on $c_1$ and $c_2$. Their values are therefore determined solely by the curved region. 
This concludes the proof.

\end{proof}

For the classification of combinations of planar and constant fold-angle creases in Section~\ref{sec:characterization}, it is more convenient to work with the integrated forms of the constraints in Theorem~\ref{thm:main} rather than the constraints themselves. We therefore note the following:

\begin{corollary}\label{cor:main}
  A candidate crease-rule pattern has a rigid-ruling folding motion if and only if for all non-zero curvature parameters $t\in T \backslash T_0$, Equation~\eqref{eqn:condAint} and Equation~\eqref{eqn:condBint} for some $c_3 > 0$ and $c_4$ are satisfied.
\end{corollary}

We highlight the following property of constant fold-angle creases:

\begin{lemma}\label{lem:allconstantfoldangle}
In a candidate crease-rule pattern that admits a rigid-ruling motion and contains at least one constant fold-angle crease, all creases are constant fold-angle creases.
\end{lemma}

\begin{proof}
For points with non-vanishing curvature, this statement follows directly from Equation~\eqref{eqn:c4}. On non-empty intervals where the curvature vanishes, the adjacent creases maintain a constant fold angle as they are folded along a straight line. For isolated parameter values with constant fold angle, Equation~\eqref{eqn:c4} does not apply. Nevertheless, the constant fold-angle property extends from one crease to the next due to the assumed continuity of the creases and ruling directions. 
\end{proof}

\section{Rigid-Ruling Foldable Crease-Rule Patterns}

Theorem~\ref{thm:main} and Corollary~\ref{cor:main} provide tools for determining whether a candidate crease-rule pattern admits a rigid-ruling folding motion. However, applying these results directly to the design of such patterns is less straightforward. In this section, we present two applications of the previous results that are suitable for design and gaining a better understanding.

\subsection{Preliminaries}\label{sec:notation_and_parametrization}

The two applications presented in this section focus on rigid-ruling foldable patterns composed of three patches. In the first application, we begin with a crease-rule pattern consisting of a single crease and append a patch along a curved crease that satisfies rigid-ruling foldability. In the second, we classify pairs of special crease types that form rigid-ruling foldable combinations.

In both applications, the approach relies on a case analysis of the central patch according to the different types of developables. Recall that the rulings of a developable surface are either locally parallel, concurrent at a point, or tangent to a space curve. Consequently, any developable surface can be decomposed into cylindrical, conical, or tangent developable patches.

Following the approach of Sauer~\cite{sauer1970differenzengeometrie}, we simplify the case analysis by employing a Combescure transformation of crease-rule patterns to convert central tangent developable patches into conical ones. Since Combescure transformations preserve the folding properties of crease-rule patterns, this simplification does not affect the analysis.

As preparation, we set up parametrizations of crease-rule patterns whose central patch is either cylindrical or conical, and express the quantities appearing in Theorem~\ref{thm:main} and Corollary~\ref{cor:main}. We continue to employ the notation of Section~\ref{sec:primer}, but simplify the common ruling direction $\vecr_{1L}(t) = \vecr_{2R}(t)$ by $\vecr(t)$.

\paragraph{Cylinders.}
In the case where the central patch is a cylinder, we assume, without loss of generality, that
\begin{align*}
\vecx_i(t) = (t,0) + l_i(t) \vecr(t) && \text{ for } && \vecr(t) = (0,1),
\end{align*}
where $l_i(t)$ are two initially unknown $C^2$ length functions such that $l_1(t) < l_2(t)$.

Using the descriptions in Section~\ref{sec:development}, we compute 
\begin{align*}
s_i'(t) k_i(t) = \frac{l_i''(t)}{1 + l_i'(t)^2}. 
\end{align*}
The ruling angles of the central patch simplify to
\begin{align*}
\theta_{1L}(t) = \arccos\left(\frac{l_1'(t)}{\sqrt{1 + l_1'(t)^2}}\right)
&& \text{and} && 
\theta_{2R}(t) = \arccos\left(\frac{l_2'(t)}{\sqrt{1 + l_2'(t)^2}}\right), 
\end{align*}
and we have that
\begin{align}
F_i(t) &= \frac{l_i''(t)}{\sqrt{1+l_i'(t)^2}}, \nonumber\\
I_1'(t) &= \frac{\left( \cot \theta_{1R}(t)+l_1'(t)\right) l_1''(t)}{2\left(1 + l_1'(t)^2\right)},\label{eqn:exprCyl}
\\
I_2'(t) &= \frac{\left( \cot \theta_{2L}(t)+l_2'(t)\right) l_2''(t)}{2\left(1 + l_2'(t)^2\right)}. \nonumber
\end{align}

If $\vecx_i(t)$ is a planar crease, then the expressions in Equation~\eqref{eqn:defI} simplify to 
\begin{align}
I_i(t) = \frac{1}{2}\log\left(1 + l_i'(t)^2\right) &&
\text{and}&&
\varphi_i(t) = \arcsin\left(c_i \sqrt{1 + l_i'(t)^2}\right), 
\end{align}
where $c_1$ and $c_2$ are appropriate constants.

\paragraph{Cones.}
In the case where the central patch is a cone, we assume, without loss of generality, that
\begin{align*}
\vecx_i(t) = \vecv + l_i(t) \left(\cos t,-\sin t\right), 
&& \text{where} && \vecr(t) =  \left(\cos t,-\sin t\right), 
\end{align*} 
and
 $l_i(t)$ are two initially unknown $C^2$ length functions such that  $l_1(t) <l_2(t)$.

Using the description in Section~\ref{sec:development}, we compute 
\begin{align*}
s_i'(t) k_i(t) = -\frac{f_i(t)}{e_i(t)}, 
\end{align*}
where 
\begin{align*}
e_i(t) &= l_i(t)^2 + l_i'(t) ^2, \\
f_i(t) &= l_i(t)^2 + 2 l_i'(t)^2 - l_i(t) l_i''(t).
\end{align*}
The ruling angles of the central patch simplify to 
\begin{align*}
\theta_{1L}(t) =  \arccos\left(\frac{l_1'(t)}{\sqrt{e_1(t)}}\right) 
&& \text{and} && 
\theta_{2R}(t) =  \arccos\left(\frac{l_2'(t)}{\sqrt{e_2(t)}}\right), 
\end{align*}
and we have that
\begin{align}
    F_i(t) &= -\frac{f_i(t)}{l_i(t) \sqrt{e_i(t)}}, \nonumber\\
    I_1'(t) &= \frac{1}{2}\left(l_1'(t) + l_1(t) \cot \theta_{1R}(t)\right)\frac{f_1(t)}{l_1(t)e_1(t)}, \label{eqn:exprCone}\\
    I_2'(t) &= \frac{1}{2}\left(l_2'(t) + l_2(t) \cot \theta_{2L}(t)\right)\frac{f_2(t)}{l_2(t)e_2(t)}. \nonumber
\end{align}

If $\vecx_i(t)$ is a planar crease, the expression in Equation~\eqref{eqn:defI} simplify to 
\begin{align*}
I_i(t) = \frac{1}{2} \log \left(\frac{e_i(t)}{l_i(t)^4}\right) 
&& \text{and} && 
\varphi_i(t) = \arcsin\left(c_i \frac{\sqrt{e_i(t)}}{l_i^2(t)}
\right),
\end{align*}
where $c_1$ and $c_2$ are appropriate constants.

Finally, note that the solutions to the differential equation $f_i(t) = 0$ are of the form
$$
l_i(t) = \frac{c_2}{\cos(t + c_1)}, 
$$
and correspond to straight lines on a cone.

\subsection{Sequential construction of rigid-ruling foldable crease-rule patterns}\label{sec:appending}

Appending a crease to a folded structure is a useful technique for increasing design complexity with additional pleats~\cite{alese2022propagation, mundilovaphd}. In general, for a given a developable patch, almost any suitable incident curve can become a curved crease, with the appended patch being found as the envelope of tangent planes reflected across the curve’s osculating plane. However, if the original surface is part of a rigidly foldable structure, in general, appending a new patch causes the structure to lose this property.

In this section, we discuss how to append a patch along a curved crease to a crease-rule pattern while preserving its rigid-ruling foldability, thereby providing a method for generating rigid-ruling foldable crease-rule patterns. Although a special case of such curves, those that have parallel tangents, was presented in Section~\ref{sec:parrul}, we show that in general these curves and corresponding patches can be specified using three parameters.

In the following, we consider the three cases for the central patch $\vecs_{1,L}(t,u)$ (cylindrical, conical, or tangent developable) separately, and show how to find a crease $\vecx_2(t)$ and the rulings of the incident surface $\vecs_{2,L}(t,u)$ such that the combined crease--rule pattern has a rigid-ruling folding motion.

\subsubsection{Appending a patch to a central cylinder}

In this section, we follow the notation of Section~\ref{sec:notation_and_parametrization} for the cylindrical case and assume that the crease $\vecx_1(t)$, defined by the length function $l_1(t)$, is not a straight line, implying $l_1'(t) \neq 0$. Using Theorem~\ref{thm:main}, we aim to find $\vecx_2(t)$, characterized by $l_2(t)$ and the rulings of the third patch, $\theta_{2L}(t)$.

First, we consider the second constraint in Theorem~\ref{thm:main} and observe that substituting  the expressions from Equation~\eqref{eqn:exprCyl} into Equation~\eqref{eqn:condB} results in 
\begin{equation}\label{eqn:theta2Lcyl}
\cot \theta_{2L}(t) = \frac{l_2''(t)}{l_1''(t)}\left(\cot \theta_{1R}(t) + l_1'(t)\right) -l_2'(t)
\end{equation}

Next, we consider the first constraint in Theorem~\ref{thm:main}.  Using Equation~\eqref{eqn:condA}, we simplify the following expression
$$
\frac{F_i'(t)}{F_i(t)} = 
-\frac{l_i'(t) l_i''(t)}{1+l_i'(t)^2} + \frac{l_i'''(t)}{l_i''(t)}.
$$
Using this expression and Equation~\eqref{eqn:theta2Lcyl},  Equation~\eqref{eqn:condA} simplifies to 
\begin{equation*}
\begin{split}
l_2''' = \frac{l_2''}{2}
&\left(
\frac{(\cot \theta_{1R} - l_1') l_1''}{1+l_1'^2} 
- \frac{(\cot \theta_{1R} +l_1')l_2''^2}{l_1''(1+l_2'^2)}
+ \frac{2 l_2' l_2''}{1+l_2'^2} 
+ \frac{2 l_1'''}{l_1''}
\right).
\end{split}
\end{equation*}
Note that this is an explicit, non-linear third-order differential equation for the length function $l_2(t)$.
Consequently, a local solution generally depends on three parameters, corresponding to the initial values of $l_2(t)$, $l_2'(t)$, and $l_2''(t)$. 
Upon successful computation, this determines $\vecx_2(t)$ and, using Equation~\eqref{eqn:theta2Lcyl}, the ruling directions of its left patch. We conclude:

\begin{corollary}
    For a given crease incident to a cylinder, there is, in general, a three-parameter family of curves on the cylinder and incident rulings that result in a crease-rule pattern with a rigid-ruling folding motion.
\end{corollary}

\subsubsection{Appending a patch to a central cone}\label{sec:appl1Cone}
Next, we consider the case of the central patch being a cone. 
We follow the notation of Section~\ref{sec:notation_and_parametrization} for the conical case and assume that the crease $\vecx_1(t)$, defined by the length function $l_1(t)$, is not a straight line, implying $f_1(t) \neq 0$. Again, using Theorem~\ref{thm:main}, we aim to find $\vecx_2(t)$, characterized by $l_2(t)$ and the rulings of the third patch, $\theta_{2L}(t)$.

Again, we first consider the second constraint in Theorem~\ref{thm:main} and observe that substituting the expressions from Equation~\eqref{eqn:exprCone} in Equation~\eqref{eqn:condB} results in 
\begin{equation}\label{eqn:theta2Lcone}
\cot \theta_{2L} = 
\frac{1}{l_2^2 }\left( - l_2l_2' + 
\frac{f_2}{f_1}\left(l_1l_1' + l_1^2 \cot \theta_{1R}\right)\right)
\end{equation}

Next, we consider the first constraint in Theorem~\ref{thm:main}.  Using Equation~\eqref{eqn:condA}, we simplify the following expression
$$
\frac{F_i'(t)}{F_i(t)} = \frac{-2 l_i'^5 + 4 l_i^3 l_i' l_i'' + 2 l_i l_i'^3 l_i'' - l_i^4 l_i''' + l_i^2 l_i'(l_i''^2 - l_i'(3 l_i' + l_i'''))}{l_i(l_i^2 + l_i'^2)(l_i^2 + 2 l_i'^2 -l_i l_i'')}
$$
Using this expression and Equation~\eqref{eqn:theta2Lcyl},  Equation~\eqref{eqn:condA} simplifies to 
\begin{equation*}
\begin{split}
    l_2''' = \frac{f_2}{2 l_2}
& %
    \left( \frac{2l_1'}{l_1} - \frac{2l_2'}{l_2} %
    + (l_1' + l_1 \cot \theta_{1R})\left(\frac{f_1}{l_1 e_1} - \frac{l_1 f_2^2}{l_2^2 e_2 f_1}\right) 
    +\frac{2l_1' (l_1 + l_1'')}{e_1}  
    -\frac{2l_2' (l_2 + l_2'')}{e_2}\right.%
    \\ %
    & \left. 
    +  \frac{-2l_1'(2 l_1 + 3 l_1'')+ 2 l_1 l_1'''}{f_1}  
+ \frac{4l_2l_2' + 6 l_2' l_2''}{f_2}
    \right).
\end{split}
\end{equation*}

Note that this is again an explicit, non-linear third-order differential equation for the length function $l_2(t)$.
Consequently, a local solution generally depends on three parameters, corresponding to the initial values of $l_2(t)$, $l_2'(t)$, and $l_2''(t)$. 
Upon successful computation, this determines $\vecx_2(t)$ and, using Equation~\eqref{eqn:theta2Lcone}, the ruling directions of its left patch. We conclude:

\begin{corollary}
    For a given crease incident to a cone, there is, in general, a three-parameter family of curves on the cone and incident rulings that result in a crease-rule pattern with a rigid-ruling folding motion.
\end{corollary}

\subsubsection{Appending a patch to a central tangent developable}

Finally, we briefly discuss the case where the central surface is a tangent developable.
As highlighted in Section~\ref{sec:notation_and_parametrization}, this case can be locally reduced to a conical case discussed in Section~\ref{sec:appl1Cone}. Specifically, we can transform the crease-rule pattern into a configuration where the surface $\tilde\vecs_{1L}(t,u)$ corresponding to $\vecs_{1L}(t,u)$ is a cone using parallelism discussed in Section~\ref{sec:parcurves}. With the method from  Section~\ref{sec:appl1Cone}, we identify the curves on $\tilde\vecs_{1L}(t,u)$ that have a rigid ruling folding motion. 
These compatible curves can be transformed into curves on $\vecs_{1L}(t,u)$, resulting in crease-rule patterns with general surfaces that have a rigid-ruling folding motion. 

 Note that no additional curves $\vecx_2(t)$ can exist, as every compatible curve on $\vecs_{1L}(t,u)$ would correspond to an unclaimed curve on 
 $\tilde \vecs_{1L}(t,u)$.  Similar to before we conclude:

 \begin{corollary}
    For a given crease incident to a tangent developable patch, there is, in general, a three-parameter family of curves on the tangent developable patch and incident rulings that result in a crease-rule pattern with a rigid-ruling folding motion.
\end{corollary}

\subsection{Rigid-ruling folding combinations of two  planar or constant fold-angle creases}\label{sec:characterization}

In this section, we explore rigid-ruling folding compatibilities of pairs of creases that are either planar or have a constant fold angle. 

The combination of two rigid-ruling foldable planar creases was studied by Mundilova and Nawratil in~\cite{mundilova2024planar} and forms the basis of our work here. They find that rigid-ruling foldable combinations of two planar creases are limited: when two creases are combined along a cylindrical patch, they may be scaled versions of each other. In all other cases, the only rigid-ruling folding compatible creases are the tangent-parallel curves discussed in Section~\ref{sec:parrul}.

In contrast, Lemma~\ref{lem:allconstantfoldangle} implies that constant fold-angle creases are compatible only with other constant fold-angle creases, which, for a given crease, can form a family of compatible creases with up to three parameters, as discussed in Section~\ref{sec:appending}, raising the question of whether this is indeed the case.

In the following section, we extend the work of Mundilova and Nawratil~\cite{mundilova2024planar} and use Corollary~\ref{cor:main} to study the implications of Lemma~\ref{lem:allconstantfoldangle} by showing:
\begin{itemize}
\item For the combination of two constant fold-angle creases, we simplify the differential equations derived in Section~\ref{sec:mainresult} and confirm that, for a given constant fold-angle crease, the second compatible crease has three degrees of freedom in Section~\ref{sec:comboconstant}.
\item When combining a constant fold-angle crease with a planar crease, Lemma~\ref{lem:allconstantfoldangle} implies that the planar crease must also be a constant fold-angle crease. Consequently, it corresponds to a curve that is perpendicular to the rulings of the developable surface. In Section~\ref{sec:comboplanarcons}, we confirm this finding.
\end{itemize}

We focus on cases where the creases are not straight, having only isolated points where $l_i''(t) = 0$ in the case of cylinders, or $f_i(t) = 0$ in the case of cones.

\subsubsection{Combination of two constant fold-angle creases}\label{sec:comboconstant}

Given a crease pattern with two constant fold-angle creases, we note that both $I_1(t) = I_2(t) = 0$. Consequently, one of the two conditions in Corollary~\ref{cor:main}, Equation~\eqref{eqn:condBint}, is trivially satisfied, while the second,  Equation~\eqref{eqn:condAint}, simplifies to
$F_1(t) = c_3 F_2(t)$.

In the following, we further simplify this constraint based on the type of the central patch..

\paragraph{Cylindrical Case.}
Using the expressions in Section~\ref{sec:notation_and_parametrization} for the case of a cylindrical patch, Equation~\eqref{eqn:condAint} simplifies to
$$
\frac{l_1''(t)}{\sqrt{1+l_1'(t)^2}}= c_3 \frac{l_2''(t)}{\sqrt{1+ l_2'(t)^2}}
$$
Using integration, we therefore conclude 
$$
\operatorname{arctanh}\left(\frac{l_1'(t)}{\sqrt{1 + l_1'(t)^2}}\right) = c_3 \operatorname{arctanh}\left(\frac{l_2'(t)}{\sqrt{1 + l_2'(t)^2}}\right)  + c_4,
$$
where $c_4$ is some constant. 

Given one of the curves, say $\vecx_2(t)$ in terms of $l_2(t)$, we can therefore compute a compatible curve from 
$$
l_1(t) = c_5 \pm \int_0^t 
\frac{\operatorname{tanh}(h(\tilde t))}{\sqrt{1 - \operatorname{tanh}^2(h(\tilde t))}}d\tilde t, 
\quad \text{ where }\quad 
h(t) = c_3 \operatorname{arctanh}\left(\frac{l_2'(t)}{\sqrt{1+ l_2'(t)^2}}\right) + c_4.
$$

In conclusion, note that for a given curve $\vecx_2(t)$, in general, we find a three-parameter family of compatible curves corresponding to the constants $c_3$, $c_4$, and $c_5$.

\paragraph{Conical case.}
On the other hand, using the expression in Section~\ref{sec:notation_and_parametrization} for the case of a central conical patch, Equation~\ref{eqn:condAint} simplifies to 
$$
 \frac{f_1(t)}{l_1(t)\sqrt{e_1(t)}}  = c_3 \frac{f_2(t)}{l_2(t) \sqrt{e_2(t)}},
$$
or equivalently, 
$$
  l_1''(t) = l_1(t) + 2 \frac{l_1'(t)^2}{l_1(t)} - c_3 \frac{\sqrt{l_1(t)^2 + l_1'(t)^2}}{l_2(t) \sqrt{l_2(t)^2 + l_2'(t)^2}}\left(l_2(t)^2 + 2 l_2'(t)^2 - l_2(t) l_2''(t)\right).
$$
While we were not able to find a closed-form solution, we expect to be able find a three-parameter family of compatible curves $l_2(t)$, where two parameters correspond to the initial values of $l_2(t)$ and $l_2'(t)$ and $c_3$ is another parameter.

\subsubsection{Combination of a planar and a constant fold-angle crease}\label{sec:comboplanarcons}

Finally, we turn our attention to crease patterns consisting of a planar crease $\vecx_1(t)$ and a crease with constant fold angle $\vecx_2(t)$. We examine how the two conditions in Corollary~\ref{cor:main} simplify in this case. Note that, since $\vecx_2(t)$ has a constant fold angle, $I_2(t) = 0$.

\paragraph{Cylindrical case.}
 Using the corresponding expressions in Section~\ref{sec:notation_and_parametrization} for a central cylindrical patch, Equation~\eqref{eqn:condAint} and Equation~\eqref{eqn:condBint} simplify to 
\begin{align*}
l_1''(t) = c_3 \frac{l_2''(t)}{\sqrt{1 + l_2'(t)^2}} && \text{ and } && l_1'(t)l_1''(t) = 0. 
\end{align*}
The second condition implies that $l_1''(t)=0$, resulting in the planar crease being straight. Consequently, there are no non-trivial combinations in this case.

\paragraph{Conical case.} Using the corresponding expressions in Section~\ref{sec:notation_and_parametrization} for a central conical patch, Equation~\eqref{eqn:condAint} and Equation~\eqref{eqn:condBint} simplify to 
\begin{align*}
\frac{f_1(t)}{l_1(t)^3} = c_3 \frac{f_2(t)}{l_2(t)\sqrt{e_2(t)}}
&&
\text{and}
&&
l_1'(t) f_1(t) = 0.
\end{align*}

Assuming that $\vecx_1(t)$ is not a straight line, that is $f_1(t)\neq 0$ except for isolated values, we conclude that $l_1'(t) = 0$ and hence $l_1(t) = \const$. Consequently, the planar crease $\vecx_1(t)$ is a curve perpendicular to the rulings of the cone, and therefore also a constant fold-angle crease.

Any folded configuration of a cone along a planar crease perpendicular to its rulings results in a right circular cone. Curves of constant fold angle on right circular cones have been studied in connection with pseudogeodesic curves, that is, curves on surfaces whose osculating planes maintain a constant angle with the incident tangent planes, by Walter Wunderlich~\cite[IV]{wunderlich1949}, where they are described via a polarization of slope lines on rotational quadrics. Parametrizations of the corresponding four types of curves in their developed form are given in \cite[VIII]{wunderlich1949} or \cite[Section 3.2]{mundilova2017}.

\paragraph{Tangent developables.} Extending the result from the conical case to tangent developables preserves the property that planar curves are those with constant fold angle, that is, perpendicular to the rulings, whereas the second type corresponds to curves parallel to pseudogeodesics on rotational cones.

\section{Conclusion and Future Work}

In conclusion, this paper studies crease-rule patterns that have a rigid-ruling motion, resulting in new findings on the conjugate-net preserving isometries of globally developable conjugate nets. 

We derived two conditions for rigid-ruling foldability and applied them to two scenarios. First, we showed that introducing a crease into a rigid-ruling foldable pattern generally introduces three degrees of freedom. Second, we investigated combinations of planar and constant fold-angle creases. In particular, we demonstrated that constant fold-angle creases are only compatible with other constant fold-angle creases. Moreover, among those that are both constant fold-angle and planar, compatibility requires that they be perpendicular to the rulings of the patch. This yields a complete characterization of compatible planar and constant fold-angle creases.

In future work, we aim to move beyond the globally developable setting and characterize all rigid-ruling–compatible combinations of developable patches. We also seek to identify and describe parallels between the smooth and discrete cases, ultimately building a bridge between smooth and discrete differential geometry and deepening our understanding of their similarities and differences.

\paragraph{Acknowledgments.} The author would like to thank Georg Nawratil for the collaboration on the compatibility of planar creases, which forms the basis of this work. She is also grateful to Tomohiro Tachi for hosting her at TachiLab during the summer of 2024, when she began working on this problem, and the Geometric Computing Laboratory at EPFL, particularly Mark Pauly. 

This research was supported by the Swiss Government Excellence Scholarship and the Swiss National Science Foundation (SNSF), grant number 200021-231293.

\bibliographystyle{plain}
\bibliography{references}

\end{document}